\documentclass[11pt]{article}

\usepackage[a4paper,margin=2.5cm]{geometry}
\usepackage{amsmath,amssymb,amsthm}
\usepackage{bm}
\usepackage{enumitem}

\usepackage{amsthm}

\usepackage{tikz}
\usetikzlibrary{calc}
\usetikzlibrary{decorations.markings, arrows.meta}

\newtheorem{definition}{Definition}[section]
\newtheorem{proposition}[definition]{Proposition}

\newtheorem{corollary}[definition]{Corollary}

\theoremstyle{remark}
\newtheorem{remark}[definition]{Remark}
\newtheorem{example}[definition]{Example}

\theoremstyle{definition}
\newtheorem{problem}[definition]{Problem}

\usepackage{hyperref}
\usepackage{url}
\hypersetup{colorlinks=true, urlcolor=blue}
\usepackage{cleveref}

\title{
Incidence-based Combinatorial Geometry on Cell Complexes
}

\author{
Andrey P. Jivkov$^{1}$\thanks{Corresponding author: andrey.jivkov@manchester.ac.uk},
Muhammad Azeem$^{1}$,
Aneirin Griffiths$^{1}$,
Pieter D. Boom$^{2}$\\[0.5em]
\small $^{1}$Department of Mechanical and Aerospace Engineering, University of Manchester,\\
\small Manchester M13 9PL, UK\\
\small $^{2}$Department of Mechanical Engineering, King Fahd University of Petroleum and Minerals,\\
\small 31261 Dhahran, Saudi Arabia
}

\date{\today}

\begin{document}
\maketitle

\begin{abstract}

Combinatorial Mesh Calculus (CMC) formulates conservation laws directly on cell complexes using combinatorial differential forms and their cochain representations. This note develops an incidence-based geometric extension of that framework for directional quantities and local geometric structure on explicit cellular organisation. The central construction is a family of local fibres generated intrinsically by the incidence structure of the complex. Each vertex carries a vector space spanned by its incident edge directions, providing a combinatorial analogue of a tangent space whose dimension reflects local topology.

These fibres define local coefficient spaces for incidence-based vector-, covector- and endomorphism-valued cochains. Directed transport maps compare states attached to neighbouring fibres, while a canonical solder form relates fibre directions to the underlying cell-complex structure. Together they give rise to finite combinatorial analogues of transport, torsion, curvature and metric structure. Evaluation cup products provide pairings between kinematic and force-like quantities, while bundle Hodge operators induced by the fibre metric relate vector- and covector-valued cochains, and weighted covariant incidence operators define degree-raising transport-corrected operations on incidence cochains.

The framework does not assume a smooth manifold, fixed-rank bundle, cellular sheaf or local system. Instead, geometric structure is generated directly from the organisation represented by the cell complex. The resulting theory establishes a combinatorial scaffold for geometry on explicit cellular organisation and identifies the principal mathematical questions required for its further development, including admissible transport classes, metric compatibility, Cartan-type structure equations and locality-dependent algebraic structures.

\end{abstract}


\section{Introduction and Motivation}
\label{sec:Introduction}


Combinatorial Mesh Calculus (CMC) was introduced as a framework for formulating conservation laws directly on cell complexes without invoking an underlying continuum description \cite{BERBATOV2022,BERBATOV2026}. In its current form, the theory is based on scalar-valued combinatorial differential forms. Differential operators arise from the incidence structure of the complex, while metric information is introduced through measures assigned to cells. This construction provides strong and weak formulations of scalar transport problems entirely within a combinatorial setting.

Many physical theories cannot be expressed solely in terms of scalar quantities. Displacement, velocity, momentum and force are directional objects. Their formulation requires local coefficient spaces, comparison of quantities attached to different locations, and geometric structures capable of representing compatibility, transport, curvature and torsion. In smooth differential geometry these requirements lead naturally to vector bundles, connections and tensor-valued differential forms \cite{Nakahara2003,Needham2021}. The corresponding constructions rely on tangent bundles and differentiable neighbourhoods and therefore do not transfer directly to explicit cell complexes.

The objective of the present work is to develop an incidence-based combinatorial geometry on cell complexes that provides intrinsically combinatorial analogues of local coefficient spaces, transport, curvature and torsion. The central idea is to derive local coefficient spaces directly from the incidence structure of the complex. Each vertex carries a vector space generated by its incident edge directions. These incidence-generated fibres provide local spaces of admissible directions and form the coefficient spaces for fibre-valued incidence cochains. Their dimension is determined by local topology and may vary from vertex to vertex. The incidence-generated and generally variable-dimensional nature of these coefficient spaces distinguishes the present construction from the fixed-rank bundles commonly used in discrete exterior calculus and related finite-element constructions \cite{Hirani2003,Desbrun2008,Arnold2006,Arnold2018}. The coefficient spaces are therefore determined by the organisation represented by the complex itself rather than prescribed independently of it. Algebraically, the resulting family of vertex spaces bears a conceptual relation to the stalk data of cellular sheaf theory \cite{Shepard1985cellular,Curry2014sheaves}. It is not, by itself, a cellular sheaf or local system: those structures require additional cellwise data and compatibility conditions. The relation is therefore used here only as motivation unless such conditions are imposed explicitly.

The physical motivation is equally important. A vector-valued 0-cochain is interpreted as a local directional state defined directly on the computational complex.  Neighbouring values belong to different fibres and in general cannot be compared directly. A transport structure is therefore required. Transport maps provide admissible comparisons between neighbouring local states and lead to transported differences and weighted degree-raising incidence operators. Through dual fibres, evaluation cup products, and bundle Hodge star operators, kinematic and force-like quantities are paired into scalar energetic fields, providing the starting point for stress-like quantities, balance laws and variational formulations on cellular structures.

The resulting framework introduces five core ingredients: a computational complex, a family of local fibres, a solder form, a directed transport structure, and a fibre metric. Weighted incidence operators and incidence inner products additionally require anchoring data and incidence weights. Curvature is introduced through oriented holonomies, torsion through transported solder forms, and degree-raising operations through weighted transport-corrected incidence sums. The resulting geometry is neither derived from a smooth manifold nor obtained by discretising a pre-existing vector bundle. Instead, local coefficient spaces, admissible directions and transport structures are generated directly from the organisation represented by the cell complex.

A central theme of the present work is locality. In conventional differential geometry locality is infinitesimal and fixed by the smooth structure. On explicit cell complexes locality becomes an independent structural choice. Transport maps, differential operators and algebraic constructions may depend on the combinatorial neighbourhood over which interactions are permitted. Consequently, locality is treated here not as a modelling parameter but as a geometric variable capable of influencing which algebraic and geometric properties become observable.

The present note is exploratory. Several constructions are proposed and analysed, but many questions remain open. Among them are the identification of admissible transport classes, metric-compatible connections, constitutive operators for the incidence-based framework, the relation between operational and differential definitions of curvature and torsion, and the emergence of locality-dependent algebraic structures. The purpose of the note is therefore twofold: to establish the current mathematical framework and to identify the principal problems governing its further development.


\section{Physical and Computational Complexes}
\label{sec:Complexes}

\subsection{Cell complexes, chains and cochains}
\label{subsec:Complexes_General}

This section recalls the algebraic structures underlying scalar Combinatorial Mesh Calculus and fixes notation used throughout the incidence-based geometric framework.

Let
\[
\mathcal{Z} = \bigcup_{p=0}^{D} Z_p
\]
be a finite oriented $D$-dimensional cell complex, where $Z_p$ denotes the set of oriented $p$-cells \cite{kozlov2008}. For cells $\sigma\in Z_p$ and $\tau\in Z_q$ with $p<q,$
\[
\sigma\prec\tau
\]
denotes the incidence relation that $\sigma$ is a face of $\tau.$ The orientation of the complex determines the boundary operator
\[
\partial_p : C_p(\mathcal{Z}) \longrightarrow C_{p-1}(\mathcal{Z}),
\]
where
\[
C_p(\mathcal{Z}) = \operatorname{span}_{\mathbb R}(Z_p)
\]
is the space of $p$-chains. The resulting chain complex is
\[
0
\longrightarrow C_D(\mathcal{Z}) \xrightarrow{\partial} \cdots \xrightarrow{\partial} C_0(\mathcal{Z}) \longrightarrow 0,
\]
with
\[
\partial^2 = 0.
\]

\begin{definition}[Cochains]
The space of $p$-cochains is
\[
C^p(\mathcal{Z}) = \operatorname{Hom}(C_p(\mathcal{Z}),\mathbb R).
\]
The transpose of the boundary operator defines the coboundary 
\[
\delta : C^p(\mathcal{Z}) \longrightarrow C^{p+1}(\mathcal{Z}),
\]
satisfying
\[
\delta^2 = 0.
\]
\end{definition}

The corresponding cochain complex is
\[
0 \longrightarrow C^0(\mathcal{Z}) \xrightarrow{\delta} \cdots \xrightarrow{\delta} C^{D}(\mathcal{Z}) \longrightarrow 0.
\]

The operators $\partial$ and $\delta$ depend only on incidence and orientation. No metric, embedding, or geometric structure is required.

\begin{remark}
The topological constructions introduced here remain unchanged throughout the combinatorial fibre framework. Connections, metrics, and fibre structures introduced later do not modify the underlying chain and cochain complexes.
\end{remark}

\subsection{Combinatorial differential forms}
\label{subsec:Complexes_Forms}

Scalar Combinatorial Mesh Calculus is formulated in terms of Forman's combinatorial differential forms \cite{Forman2002}.

\begin{definition}[Combinatorial differential form]
A combinatorial differential $q$-form on $\mathcal{Z}$ is a linear map
\[
\omega : C_\bullet(\mathcal{Z}) \longrightarrow C_\bullet(\mathcal{Z})
\]
of degree $-q.$ Thus
\[
\omega : C_p(\mathcal{Z}) \longrightarrow C_{p-q}(\mathcal{Z}), \qquad p\ge q.
\]
For every cell $\tau\in Z_p,$
\[
\omega(\tau) = \sum_{\substack{\sigma\prec\tau\\ 
\sigma\in Z_{p-q}}} \omega^\tau_\sigma\,\sigma.
\]
\end{definition}
The coefficients vanish whenever $\sigma$ is not incident with $\tau.$ Combinatorial differential forms are therefore intrinsically local: their action depends only on the incidence structure of the complex.

The graded vector space of combinatorial forms is
\[
\Omega^\bullet(\mathcal{Z}) = \bigoplus_{q=0}^{D} \Omega^q(\mathcal{Z}).
\]
The combinatorial exterior derivative
\[ d : \Omega^q(\mathcal{Z}) \longrightarrow \Omega^{q+1}(\mathcal{Z})
\]
is defined by
\[
d\omega = \omega\circ\partial - (-1)^q \partial\circ\omega.
\]
It satisfies
\[
d^2 = 0.
\]
Consequently,
\[
(\Omega^\bullet(\mathcal{Z}),d)
\]
forms a differential complex.

\begin{remark}
The operator \(d\) is the fundamental differential operator of scalar CMC. Later sections introduce a weighted degree-raising operator on fibre-valued incidence cochains; no combinatorial fibre extension of Forman's differential-form space is assumed.
\end{remark}

\subsection{Physical complexes}
\label{subsec:PhysicalComplex}

The starting point of the present framework is an explicit representation of physical organisation.

\begin{definition}[Physical complex]
Let
\[
\mathcal{M}=\bigcup_{p=0}^{D}M_p
\]
be a finite oriented $D$-dimensional cell complex whose cells are simple polytopes. The complex $\mathcal{M}$ is called a \emph{physical complex} if its cells represent physically distinguishable entities, interfaces, junctions, interaction domains, or other structural components of a material system.
\end{definition}

The cells of $\mathcal{M}$ are therefore not introduced as approximation domains for a continuum field. Instead, they constitute the organisation of a physical system. The incidence relations of the complex describe organisation explicitly through adjacency, connectivity, containment, and neighbourhood structure.

The physical interpretation attached to a cell depends on the particular system under consideration. A $D$-cell may represent an individual grain, particle, pore, biological compartment, network domain, or structural component \cite{Fratzl2007,Ellenbroek2015}. Lower-dimensional cells represent interfaces and higher-order interaction structures arising from the incidence relations of the physical fabric.

The distinction between a physical complex and a conventional computational mesh is fundamental. In standard discretisation approaches, cells are numerical devices used to approximate fields defined on an underlying continuum. In contrast, the cells of $\mathcal{M}$ possess direct physical meaning and remain part of the description at all stages of the theory.

\begin{remark}[Organisation as a physical observable]
The primary information carried by $\mathcal{M}$ is organisational. Its cells identify physical entities, while its incidence structure determines admissible relations between them \cite{Mariano2024}. Quantities derived from connectivity, neighbourhood structure, dimensional hierarchy and locality become candidate physical descriptors; they are not auxiliary properties of a discretisation.
\end{remark}

The physical complex provides the topological foundation of the incidence-based geometric framework developed later. Local fibres, transport maps, solder forms and covariant differential operators will all be derived from the explicit organisational information contained in $\mathcal{M}.$ In particular, the local fibres introduced in subsequent sections are constructed from the incidence structure of $\mathcal{M}$ and therefore represent admissible local directions intrinsic to the particular organisation. This construction replaces the tangent-space concept of smooth differential geometry by a combinatorial counterpart generated directly from physical organisation.

\subsection{Computational complexes}
\label{subsec:ComputationalComplex}

The geometric framework developed in this work is formulated on a computational complex derived from the physical complex.

\begin{definition}[Computational complex]
Let $\mathcal{M}$ be a physical complex. The associated computational complex $\mathcal{K}$ is the Forman subdivision of $\mathcal{M},$ whose cells represent incidences between cells of the physical complex.
\end{definition}

Under the assumption that the physical complex consists of simple polytopes, the Forman subdivision is a cubical complex and therefore inherits a natural dimensional and incidence structure suitable for combinatorial differential operations.

The significance of the computational complex follows from Forman's identification of combinatorial differential forms with cochains on the subdivision complex \cite{Forman2002}.

\begin{proposition}[Forman correspondence]
Let
\[
\Omega^\bullet(\mathcal{M})
\]
denote the graded space of combinatorial differential forms on the physical complex and
\[
C^\bullet(\mathcal{K})
\]
the graded cochain space of the computational complex. The Forman subdivision identifies the coefficients of combinatorial differential forms on $\mathcal{M}$ with cochains on $\mathcal{K}$, giving a natural isomorphism 
\[
\Omega^\bullet(\mathcal{M}) \cong C^\bullet(\mathcal{K}).
\]
Under this correspondence, the combinatorial exterior derivative 
\[
d: \Omega^p(\mathcal{M}) \rightarrow \Omega^{p+1}(\mathcal{M})
\]
is represented exactly by the coboundary operator
\[
\delta: C^p(\mathcal{K}) \rightarrow C^{p+1}(\mathcal{K}).
\]
Consequently,
\[
d^2=0 \qquad\Longleftrightarrow\qquad \delta^2=0.
\]
\end{proposition}

The correspondence permits all combinatorial differential operations to be expressed algebraically in terms of coboundaries on $\mathcal{K}.$ This observation constitutes the foundation of scalar Combinatorial Mesh Calculus and remains fundamental for the combinatorial fibre extension developed later.

In the present framework, physical interpretation is attached to the cells of $\mathcal{M},$ whereas algebraic operations are carried out on the cochains of $\mathcal{K}.$ The two complexes therefore play complementary roles:
\begin{itemize}
\item $\mathcal{M}$ describes the physical organisation and its constituent entities and relations;
\item $\mathcal{K}$ provides the computational and algebraic structure required for differential operations.
\end{itemize}

\paragraph{Cup product.}

The cubical structure of the computational complex permits the construction of a cup product \cite{Wilson2007,Arnold2012},
\[
\smile : C^p(\mathcal{K})\times C^q(\mathcal{K}) \longrightarrow C^{p+q}(\mathcal{K}).
\]
The product is assembled from pairs of topologically orthogonal subcells sharing a unique common 0-cell. The contribution associated with a $(p+q)$-cell is obtained by summing appropriately weighted products of the cochain values on these perpendicular subcells. The precise construction is given in \cite{BERBATOV2026}.

The cup product is compatible with the coboundary operator through the graded Leibniz identity
\[
\delta(\alpha\smile\beta) = \delta\alpha\smile\beta + (-1)^p \alpha\smile\delta\beta.
\]

Through Forman's correspondence, the cup product induces a corresponding product on combinatorial differential forms on the physical complex, yielding
\[
(\Omega^\bullet(\mathcal{M}),d,\wedge) \cong (C^\bullet(\mathcal{K}),\delta,\smile).
\]

The cup product plays a central role in scalar CMC. In particular, it enters the construction of discrete Hodge-star operators and variational formulations. 

\begin{remark}
Unlike the smooth exterior product, the cubical cup product used in scalar CMC is generally not associative. 
\end{remark}

\begin{remark}[Transition to bundle coefficients]
In the scalar theory, the cup product pairs real-valued cochains by multiplying scalar coefficients evaluated at the common $0$-cell $N(b,c)$ of topologically orthogonal subcells $b \perp_\chi c$ \cite{BERBATOV2026}. In Section~\ref{sec:BundleValued}, this stencil is directly generalized to the evaluation cup product $\smile_{ev},$ where scalar multiplication is replaced by the canonical evaluation pairing between dual fibres $F_a^*$ and vector fibres $F_a.$
\end{remark}

\subsection{Measures and Metric-Dependent Operators}
\label{subsec:Complexes_Metric}

The operators introduced so far depend only on incidence and orientation. Metric-dependent constructions are obtained by assigning positive values to cells.

\begin{definition}[Cell measures]
Let
\[
\mu : \mathcal{K} \rightarrow \mathbb R_{>0}
\]
be a function assigning positive values to the cells of the computational complex. For each cell
\[
\sigma\in\mathcal{K},
\]
the value
\[
\mu(\sigma)
\]
is called its measure.
\end{definition}

The measure function introduces quantitative information while remaining independent of the incidence structure. Its values may be chosen according to the application and need only satisfy positivity.

The construction of metric-dependent operators requires the notion of topological orthogonality.
\begin{definition}[Topological orthogonality]
Let
\[
\chi\in\mathcal{K}_D, \qquad \sigma\in \mathcal{K}_p, \qquad
\tau\in \mathcal{K}_{D-p}.
\]
The cells $\sigma$ and $\tau$ are said to be
\emph{topologically orthogonal} with respect to $\chi,$
\[
\sigma \perp_\chi \tau,
\]
if both are faces of $\chi$ and intersect in a single common
$0$-cell.
\end{definition}

Topologically orthogonal cells play the combinatorial role of complementary orthogonal faces of the same top-dimensional cell \cite{BERBATOV2026}.

\paragraph{Inner products.}

The measure structure induces inner products on cochain spaces. In the scalar CMC, basis cochains are chosen to be orthogonal. For
\[
\sigma\in \mathcal{K}_p,
\]
the induced basis norm is
\[
\langle \sigma^\bullet,\sigma^\bullet\rangle_p =
\frac{1}{2^D\mu(\sigma)} \sum_{\substack{\chi\in\mathcal{K}_D\\ \tau\in \mathcal{K}_{D-p}\\ \sigma\perp_\chi\tau}} \mu(\tau).
\]
This construction defines the inner product
\[
\langle\cdot,\cdot\rangle_p: C^p(\mathcal{K})\times C^p(\mathcal{K}) \rightarrow \mathbb R.
\]
Its physical dimension is 
\[
[\langle\cdot,\cdot\rangle_p] = L^{D-2p},
\]
in agreement with the continuum theory \cite{BERBATOV2026}.

\paragraph{Hodge-star operators.}

The Hodge-star operator is defined from the cochain inner products and the cup product.

\begin{proposition}[Discrete Hodge star]
For
\[
0\le p\le D,
\]
there exists a unique operator
\[
\star_p : C^p(\mathcal{K}) \rightarrow C^{D-p}(\mathcal{K})
\]
satisfying
\[
\langle \star_p\sigma, \tau \rangle_{D-p} = (\sigma\smile\tau)[\mathcal{K}], \]
for all
\[
\sigma\in C^p(\mathcal{K}), \qquad \tau\in C^{D-p}(\mathcal{K}).
\]
\end{proposition}

The quantity
\[
(\sigma\smile\tau)[\mathcal{K}]
\]
denotes evaluation of the top-dimensional cochain \(\sigma\smile\tau\) on the fundamental chain of the oriented computational complex.

Through Forman's correspondence, the Hodge stars are transferred to combinatorial differential forms on the physical complex.

\begin{remark}[Relation to bundle Hodge operators]
The scalar Hodge star $\star_p$ converts $p$-cochains to $(D-p)$-cochains while preserving the scalar coefficient field $\mathbb{R}.$ In the combintorial fibre framework (Section~\ref{sec:Metrics}), the fibre metric $g_a$ introduces a musical isomorphism between $F_a$ and its dual $F_a^*.$ As a result, the discrete Hodge star splits into two distinct operational maps: a vector-to-covector Hodge star $\star_p^\flat$ (transforming kinematic fields into force/stress fields) and a covector-to-vector Hodge star $\star_p^\sharp.$
\end{remark}

\paragraph{Adjoint coboundaries.}

Let
\[
C^p_0(\mathcal{K}) \subset C^p(\mathcal{K})
\]
denote the subspace of cochains satisfying homogeneous essential boundary conditions. The restriction of the coboundary operator is
\[
\widetilde\delta_{p-1} : C^{p-1}_0(\mathcal{K})
\rightarrow C^{p}_0(\mathcal{K}).
\]

\begin{definition}[Adjoint coboundary]
The adjoint coboundary
\[
\delta^\star_p: C^p_0(\mathcal{K}) \rightarrow C^{p-1}_0(\mathcal{K})
\]
is defined by
\[
\langle \tau, \widetilde\delta_{p-1}\sigma \rangle_p =
\langle \delta^\star_p\tau, \sigma \rangle_{p-1},
\]
for all
\[
\tau\in C^p_0(\mathcal{K}), \qquad \sigma\in C^{p-1}_0(\mathcal{K}).
\]
\end{definition}

Under the correspondence between combinatorial differential forms and computational cochains, this operator represents the codifferential of the combinatorial differential form calculus.

\paragraph{Constitutive information.}

The measure function provides only the geometric component of the metric-dependent operators. In physical applications additional constitutive information is typically associated with incidences rather than individual cells.

For example, diffusivities, conductivities, mobilities, and elastic interaction coefficients are naturally associated with incidences between adjacent 1-cells and 2-cells and therefore modify $\star_1$ and $\star_{D-1}.$

Similarly, storage coefficients, volumetric capacities, and compressibilities are naturally associated with incidences between adjacent 0-cells and $D$-cells and therefore modify $\star_0$ and $\star_D.$

Consequently, discrete Hodge-star operators generally contain both geometric information derived from cell measures and constitutive information describing physical interactions between neighbouring entities.

\paragraph{Summary.}

The metric-dependent part of scalar CMC consists of:
\begin{itemize}
\item positive cell measures;
\item cochain inner products;
\item Hodge-star operators;
\item adjoint coboundaries.
\end{itemize}
These structures depend on the measure function and on the
previously defined topological operations
\[
\partial,\qquad \delta,\qquad \smile,
\]
but they do not modify the underlying chain, cochain, or cup-product structures.

Up to this point, the scalar theory has been recalled without modification. The remainder of this note develops the incidence-based geometric extension obtained by replacing scalar coefficients with locally varying fibres and by introducing the associated transport, metric and curvature structures.

\section{Incidence-generated Fibres and Cochains}
\label{sec:BundleValued}

\subsection{Local coefficient spaces}
\label{subsec:BundleValued_Motivation}

Scalar-valued cochains associate a single numerical value with each cell. This is sufficient for conservation laws involving scalar quantities but does not provide local state spaces for directional, force-like, or internally structured quantities.

The framework developed below preserves the underlying cochain structure of scalar CMC while replacing the fixed scalar coefficient field by a family of incidence-generated local vector spaces attached to the vertices of the computational complex.

\subsection{Local fibres}
\label{subsec:BundleValued_Fibres}

Let
\[
a\in \mathcal{K}_0.
\]
Its incident 1-cell star is
\[
\operatorname{St}_1(a) = \{\,e\in \mathcal{K}_1 : a\prec e\,\}.
\]
The closed vertex neighbourhood of \(a\) is
\[
N(a) = \{\,b\in \mathcal{K}_0 : a=b \text{ or } \exists e\in \mathcal{K}_1,\ a\prec e,\ b\prec e\}.
\]
Its set of adjacent vertices is
\[
\operatorname{Adj}(a)=N(a)\setminus\{a\}.
\]
Thus \(N(a)\) consists of \(a\) together with the vertices joined to it by an incident \(1\)-cell. Because \(\mathcal{K}\) is a cubical complex, every \(1\)-cell incident with \(a\) has a unique other endpoint \(b\in\operatorname{Adj}(a)\). We denote the corresponding fibre basis element by
\[
\mathbf e_{a,b}.
\]
The degree of the 0-cell is
\[
n_a = \#\operatorname{St}_1(a)=\#\operatorname{Adj}(a).
\]

\begin{definition}[Local fibre]
The fibre associated with the vertex \(a\) is
\[
F_a = \operatorname{span}_{\mathbb R}
\{\,\mathbf e_{a,b}:b\in\operatorname{Adj}(a)\,\}.
\]
\end{definition}

The basis vectors are indexed by the oriented incidences leaving the reference vertex. More precisely, for every edge joining the vertices
\(a\) and \(b\), the basis element
\[
\mathbf e_{a,b}\in F_a
\]
represents the oriented edge directed from \(a\) towards \(b\). Accordingly,
\[
\mathbf e_{a,b}\in F_a, \qquad \mathbf e_{b,a}\in F_b,
\]
are distinct basis vectors belonging to different fibres. They are not related by a sign change, since they represent different local directions attached to different vertices. Consequently,
\[
F_a \simeq \mathbb R^{n_a}, \qquad \dim F_a=n_a.
\]
Every element
\[
X_a\in F_a
\]
admits the unique expansion
\[
X_a = \sum_{b\in\operatorname{Adj}(a)} X_a^{\,b}\, \mathbf e_{a,b}.
\]
The collection
\[
F = \{F_a\}_{a\in \mathcal{K}_0} 
\]
forms a family of local vector spaces indexed by the vertices of the computational complex. The coefficient spaces are therefore determined by the organisation represented by the complex itself rather than prescribed independently of it.

\begin{remark}
The fibres are generated directly from the local incidence structure of the complex. Unlike tangent spaces in smooth differential geometry, they do not rely on an ambient manifold, coordinates, or a differentiable structure. Incident 1-cells are chosen because they represent the minimal combinatorial notion of local direction. The orientations used to define the fibre bases are local and intrinsic. Every basis vector is directed away from its reference vertex towards the opposite endpoint of the corresponding edge. These local orientations are independent of the global orientations assigned to the $1$-cells in the chain and cochain complexes. The latter enter only through the incidence numbers appearing in boundary and coboundary operators.
\end{remark}

\begin{remark}[Relation to cellular sheaves]
In cellular sheaf theory \cite{Shepard1985cellular,Curry2014sheaves}, vector spaces are assigned to cells and are connected by restriction maps satisfying functorial compatibility. The family $\{F_a\}_{a\in \mathcal{K}_0}$ provides only the vertex-level vector-space data from which such a structure might be developed. It does not by itself define a cellular sheaf or a local system. Its distinctive feature is that each \(F_a\) is generated by \(\operatorname{St}_1(a)\), so that
\[
\dim F_a=n_a
\]
is tied directly to local valence and may vary with \(a\).
\end{remark}

\subsection{Sections}
\label{subsec:BundleValued_Sections}

A section of the fibre family is a map
\[
X : a \longmapsto X(a)\in F_a.
\]
The vector space of all sections is denoted
\[
C^0(\mathcal{K}; F).
\]
Locally,
\[
X(a) = \sum_{b\in\operatorname{Adj}(a)} X_a^{\,b} \mathbf e_{a,b}.
\]
A section assigns a local state to every vertex of the computational complex. Values of a section belong to different fibres and cannot in general be compared directly. The comparison of neighbouring states requires additional geometric structures that will be introduced later.

\subsection{Fibre-valued incidence cochains}
\label{subsec:BundleValued_Cochains}

For every \(p\ge0\), define the set of incidences
\[
\mathcal{K}_{0,p} = \{ (a,\sigma): a\in \mathcal{K}_0,\; \sigma\in \mathcal{K}_p,\; a\prec\sigma\}.
\]
Let
\[
V=\{V_a\}_{a\in \mathcal{K}_0}
\]
be a family of vector spaces indexed by the 0-cells.

\begin{definition}[Fibre-valued incidence cochain]
A \(V\)-valued \(p\)-cochain is a map
\[
\omega : \mathcal{K}_{0,p} \rightarrow V,
\]
such that
\[
\omega(a,\sigma)\in V_a
\]
for every incident pair \((a,\sigma)\).
\end{definition}

The coefficient space is therefore determined by the incident 0-cell, whereas the cochain degree is determined by the dimension of the associated cell. As in the scalar theory, fibre-valued incidence cochains are assumed to respect the orientation of the supporting cell. If the orientation of a \(p\)-cell for $p>0$ is reversed, then the corresponding coefficient changes sign. Thus, for every incident pair,
\[
(a,\sigma)\in\mathcal{K}_{0,p},
\]
one requires
\[
\omega(a,-\sigma) = -\omega(a,\sigma).
\]
For \(p=0\), this convention reduces to the ordinary assignment of fibre values to vertices. Accordingly, fibre-valued incidence cochains remain alternating objects with respect to the orientation of the underlying cells. The additional fibre structure changes only the coefficient space attached to each incidence; it does not alter the orientation conventions inherited from the scalar cochain complex.

\begin{remark}[Motivation for fibre-valued incidence cochains]
Notice that an vector-valued $p$-cochain assigns a vector to every incident pair $(a,\sigma)$ with $a\prec\sigma,$ rather than a single vector to the cell $\sigma$ alone. In standard cochain theories with coefficients in a fixed space $V,$ a cochain simply evaluates directly on cells via $X: \mathcal{K}_p \to V.$ In the present framework, this cell-based assignment is insufficient for two main reasons:
\begin{enumerate}
    \item \textbf{Fibre Inhomogeneity:} The local fibres $F_a$ are intrinsically generated by local vertex stars and vary across vertices ($F_a \neq F_b$ whenever $\operatorname{deg}(a) \neq \operatorname{deg}(b)$). A higher-dimensional cell $\sigma \in \mathcal{K}_p$ ($p \ge 1$) does not possess a single canonical intrinsic fibre of its own.
    \item \textbf{Local Degrees of Freedom and Anchoring for Transport:} Evaluating at the incidence $(a, \sigma)$ anchors the vector selection to $F_a$ while maintaining $\sigma$ as the geometric domain. This provides the necessary coefficient structure for transport maps $P$ and covariant differentiation, while allowing the theory to represent localised, directional, or discontinuous physical quantities independently at each vertex of $\sigma.$
\end{enumerate}
\end{remark}

\subsection{Vector-valued cochains}
\label{subsec:BundleValued_Vectors}

Choosing
\[
V_a= F_a
\]
gives the fibre-valued incidence cochain space
\[
C^p(\mathcal{K};F).
\]
Every value admits the expansion
\[
X(a,\sigma) = \sum_{b\in\operatorname{Adj}(a)} X_{a,\sigma}^{\,b} \,\mathbf e_{a,b}.
\]
For \(p=0\),
\[
C^0(\mathcal{K};F) 
\]
coincides with the space of sections introduced above. The orientation convention implies
\[
X(a,-\sigma)=-X(a,\sigma).
\]

\subsection{Dual fibres and covector-valued cochains}
\label{subsec:BundleValued_Covectors}

The dual fibre at the vertex \(a\) is
\[
F_a^* = \operatorname{Hom}(F_a,\mathbb R).
\]
Let
\[
\{\mathbf e^{a,b}\}
\]
denote the dual basis satisfying
\[
\mathbf e^{a,b} (\mathbf e_{a,c}) = \delta_{bc}.
\]
The resulting cochain spaces are
\[
C^p(\mathcal{K};F^*).
\]
Their elements admit the local representation
\[
\eta(a,\sigma) = \sum_{b\in\operatorname{Adj}(a)} \eta_{a,\sigma}^{\,b}
\mathbf e^{a,b}.
\]
The pairing
\[
F_a^* \times F_a \rightarrow \mathbb R
\]
is defined independently of any metric structure. Additionally,
\[
\eta(a,-\sigma)=-\eta(a,\sigma).
\]

\begin{remark}
The distinction between fibres and dual fibres is fundamental in geometric theories that separate directional and force-like quantities. Kinematic quantities belong naturally to fibres, whereas force-like quantities belong naturally to dual fibres. 
\end{remark}

\subsection{Evaluation cup product}
\label{subsec:BundleValued_Evaluation}

The canonical pairing
\[
F_a^*\times F_a \longrightarrow \mathbb R, \qquad (\eta,X)\mapsto\eta(X),
\]
induces a scalar-valued product between covector-valued and vector-valued cochains.

The construction uses the same stencil as the scalar cup product. Only the multiplication of scalar coefficients is replaced by evaluation of a vector by a covector. Because perpendicular subcells share a common 0-cell, the associated coefficients belong to the same fibre and can therefore be paired without transport.

\begin{definition}[Evaluation cup product]
Let
\[
\eta\in C^p(\mathcal{K};F^*), \qquad X\in C^q(\mathcal{K};F),
\]
and let
\[
a\in \mathcal{K}_{p+q}.
\]
For every perpendicular pair
\[
(b,c)\in \perp_{p,q}a,
\]
where $\perp_{p,q}a$ is defined in \cite{BERBATOV2026}, let
\[
N(b,c)\in \mathcal{K}_0
\]
denote the unique common \(0\)-cell of \(b\) and \(c\). The evaluation cup product is the bilinear map
\[
\smile_{\mathrm{ev}}: C^p(\mathcal{K};F^*) \times C^q(\mathcal{K};F) \longrightarrow C^{p+q}(\mathcal{K})
\]
defined by
\[
(\eta\smile_{\mathrm{ev}}X)(a_\bullet) = \frac{1}{2^{p+q}} \sum_{(b,c)\in\perp_{p,q}a} \operatorname{rel}(a,b,c)\, \eta(N(b,c),b) \left(X(N(b,c),c)
\right),
\]
where $\operatorname{rel}(a,b,c)$ is defined in \cite{BERBATOV2026}.
\end{definition}
Thus the evaluation cup product is obtained from the scalar cubical cup product by replacing the coefficient product
\[
\sigma(b_\bullet)\tau(c_\bullet)
\]
with the fibre evaluation
\[
\eta(N(b,c),b) \left(X(N(b,c),c)\right).
\]
All orientations, signs, weights, and perpendicular-cell pairs are inherited from the scalar cup product. Consequently,
\[
\eta\smile_{\mathrm{ev}}X
\]
is an ordinary scalar-valued \((p+q)\)-cochain. Its value is attached to the \((p+q)\)-cell \(a\), not to an incidence.

If
\[
X\in C^p(K;F), \qquad \eta\in C^q(K;F^*),
\]
then the reversed evaluation cup product is defined by
\[
(X\smile_{\mathrm{ev}}\eta)(a_\bullet) = \frac{1}{2^{p+q}}\sum_{(b,c)\in\perp_{p,q}a}\operatorname{rel}(a,b,c)\, \eta(N(b,c),c) \left(X(N(b,c),b)\right).
\]

\paragraph{Products involving \(0\)-cochains.}
If
\[
\eta\in C^0(\mathcal{K};F^*), \qquad X\in C^q(\mathcal{K};F),
\]
then for \(a\in \mathcal{K}_q\),
\[
(\eta\smile_{\mathrm{ev}}X)(a_\bullet) =
\frac{1}{2^q} \sum_{N\prec a} \eta(N) \left( X(N,a)\right).
\]
Similarly, if
\[
\eta\in C^p(\mathcal{K};F^*), \qquad X\in C^0(\mathcal{K};F), 
\]
then for \(a\in \mathcal{K}_p\),
\[
(\eta\smile_{\mathrm{ev}}X)(a_\bullet) = \frac{1}{2^p} \sum_{N\prec a} \eta(N,a) \left(
X(N)\right).
\]

\paragraph{Example: \(1\)- and \(1\)-cochains.}

For
\[
\eta\in C^1(\mathcal{K};F^*), \qquad X\in C^1(\mathcal{K};F),
\]
and \(f\in \mathcal{K}_2\),
\[
(\eta\smile_{\mathrm{ev}}X)(f_\bullet) = \frac{1}{4} \sum_{(e_i,e_j)\in\perp_{1,1}f}
\operatorname{rel}(f,e_i,e_j)\, \eta(N(e_i,e_j),e_i) \left(
X(N(e_i,e_j),e_j)\right).
\]

\paragraph{Example: \(2\)- and \(1\)-cochains in dimension three.}

For
\[
\tau\in C^2(\mathcal{K};F^*), \qquad u\in C^1(\mathcal{K};F),
\]
and \(c\in \mathcal{K}_3\),
\[
(\tau\smile_{\mathrm{ev}}u)(c_\bullet) = 
\frac{1}{8} \sum_{(f,e)\in\perp_{2,1}c} \operatorname{rel}(c,f,e)\,
\tau(N(f,e),f) \left(u(N(f,e),e) \right).
\]
Thus
\[
\tau\smile_{\mathrm{ev}}u \in C^3(\mathcal{K}),
\]
which may be interpreted as a scalar-valued power, work, energy, or dissipation cochain, depending on the physical meaning of
\(\tau\) and \(u\).

\begin{remark}
The ordinary cup product is order-sensitive at the level of the supporting cochains. Therefore, when the vector-valued factor is placed first and the covector-valued factor second, the same cubical cup-product order is retained. Only the fibre coefficient is evaluated by applying the covector to the vector.
\end{remark}

\begin{remark}
The evaluation cup product does not require transport maps or metric structures. The two coefficients in each summand are evaluated in the same fibre because the perpendicular subcells \(b\) and \(c\) share the node \(N(b,c)\). Metric information enters only through fibre metrics and fibre Hodge operators introduced later.
\end{remark}

\subsection{Endomorphism-valued cochains}
\label{subsec:BundleValued_Endomorphisms}

For every vertex
\[
a\in\mathcal{K}_0,
\]
let
\[
\operatorname{End}(F_a)=\operatorname{Hom}(F_a,F_a)
\]
denote the vector space of linear endomorphisms of the local fibre. The corresponding fibre-valued incidence cochains form
\[
C^p(\mathcal{K};\operatorname{End}(F)).
\]

The principal example of an endomorphism-valued cochain developed in the present work is the curvature endomorphism introduced in Section~\ref{sec:Curvature}. More generally, these spaces provide the natural setting for fibrewise linear operators acting on the local fibres.

As for vector- and covector-valued cochains, orientation is carried by the supporting cell. Thus
\[
A(a,-\sigma)=-A(a,\sigma),
\]
where the minus sign is taken in the vector space
\[
\operatorname{End}(F_a).
\]

In the constructions developed here, endomorphisms always act on the reference fibre associated with the corresponding incidence, and no transport of endomorphisms is required.

\subsection{Fibrewise algebra}
\label{subsec:BundleValued_Algebra}

Since all values associated with a fixed incidence belong to the same fibre, linear algebraic operations may be performed pointwise. For example,
\[
(X+Y)(a,\sigma) = X(a,\sigma)+Y(a,\sigma),
\]
and
\[
(\lambda X)(a,\sigma) = \lambda X(a,\sigma),
\]
define elements of \(C^p(\mathcal K;F)\).

Products require more care because of orientation parity. For a family of coefficient spaces \(W=\{W_a\}\), let
\[
\mathcal I^p(\mathcal K;W)
\]
denote the vector space of all incidence functions
\[
\phi:\mathcal K_{0,p}\longrightarrow
\bigsqcup_{a\in\mathcal K_0}W_a,
\qquad
\phi(a,\sigma)\in W_a,
\]
without imposing an orientation rule. For \(p>0\), define the parity subspaces
\[
\mathcal I^p_\pm(\mathcal K;W)
=
\left\{
\phi\in\mathcal I^p(\mathcal K;W):
\phi(a,-\sigma)=\pm\phi(a,\sigma)
\right\}.
\]
Thus
\[
C^p(\mathcal K;F)
=\mathcal I^p_-(\mathcal K;F),
\qquad
C^p(\mathcal K;F^*)
=\mathcal I^p_-(\mathcal K;F^*)
\]
for \(p>0\).

If
\[
X\in C^p(\mathcal K;F),
\qquad
\eta\in C^p(\mathcal K;F^*),
\]
their pointwise evaluation
\[
\eta(X)(a,\sigma) = \eta(a,\sigma) \left(X(a,\sigma)\right)
\]
satisfies
\[
\eta(X)(a,-\sigma)=\eta(X)(a,\sigma).
\]
It therefore defines an orientation-even scalar incidence function,
\[
\eta(X)\in\mathcal I^p_+(\mathcal K;\mathbb R),
\]
not an ordinary scalar cochain in \(C^p(\mathcal K)\), which has one alternating value per \(p\)-cell.

Likewise,
\[
(X\otimes\eta)(a,\sigma) = X(a,\sigma)\otimes\eta(a,\sigma)
\]
defines an element of
\[
\mathcal I^p_+
\bigl(\mathcal K;F\otimes F^*\bigr).
\]
By contrast, if
\[
A\in C^0(\mathcal K;\operatorname{End}(F)),
\qquad
X\in C^p(\mathcal K;F),
\]
then the pointwise action
\[
(AX)(a,\sigma)=A(a)X(a,\sigma)
\]
is alternating and hence belongs to \(C^p(\mathcal K;F)\).

\begin{remark}
All operations in this subsection are local and require neither transport maps nor metric structures. The distinction between ordinary scalar cochains and incidence functions is nevertheless essential. Aggregating an incidence function into an ordinary cochain requires an additional operation, such as the evaluation cup product or an explicitly weighted projection.
\end{remark}

\begin{remark}[Algebraic operations]
The fibrewise operations act independently at each incidence \((a,\sigma)\). Products of two alternating incidence fields are orientation-even, whereas the action of an orientation-independent \(0\)-cochain preserves the parity of the field on which it acts. This bookkeeping replaces the unsupported identification of pointwise products with ordinary scalar cochains.
\end{remark}

\section{Connections and Transport}
\label{sec:Connections}

\subsection{Need for transport}
\label{subsec:Connections_Need}

Fibre-valued incidence cochains assign values to locally varying coefficient spaces. Consequently, values associated with neighbouring vertices belong to different fibres and cannot, in general, be compared directly.

For scalar-valued cochains this difficulty does not arise, since all quantities belong to the same coefficient field. In the incidence-valued setting, however, the comparison of neighbouring local states requires additional structure.

The purpose of the transport structure is to provide such comparisons. It does so by assigning linear maps between neighbouring fibres. Additional hypotheses determine when this structure qualifies as reversible parallel transport.

\subsection{Directed edge transports}
\label{subsec:Connections_Definition}

Let
\[
e=(a,b)\in \mathcal{K}_1
\]
be an oriented 1-cell joining the vertices
\[
a,b\in \mathcal{K}_0.
\]

\begin{definition}[Directed edge-transport structure]
A directed edge-transport structure is a family
\[
P= \{P_{ab}\}_{(a,b)\in \mathcal{K}_1},
\]
where each oriented 1-cell is assigned a linear map
\[
P_{ab} : F_a \longrightarrow F_b.
\]
\end{definition}

The maps \(P_{ab}\) are called directed transport maps. If the opposite orientation is represented by
\[
(b,a)\in \mathcal{K}_1,
\]
then
\[
P_{ba} : F_b \longrightarrow F_a
\]
is, at this level of generality, independent of \(P_{ab}\). This definition permits transport between fibres of unequal dimensions and includes non-injective and non-surjective maps. 

\begin{definition}[Reversible discrete connection]
A directed edge-transport structure is called a reversible discrete connection if, for every unoriented edge \(\{a,b\}\), the map
\[
P_{ab}:F_a\longrightarrow F_b
\]
is an isomorphism and
\[
P_{ba}=P_{ab}^{-1}.
\]
\end{definition}

A reversible discrete connection can therefore exist along the edge \(\{a,b\}\) only if
\[
\dim F_a=\dim F_b,
\qquad\text{equivalently}\qquad
n_a=n_b.
\]
On complexes with varying vertex valence, the general directed transport structure is consequently the primary object. Reversibility may be imposed only on equal-rank subcomplexes or after replacing the incidence-generated fibres by suitable common-rank subspaces, quotients, or extensions. A reversible discrete connection is therefore a special class of directed transport structures rather than the default transportation mechanism.

\begin{remark}
The reversible case is the closest finite analogue of parallel transport in a vector bundle. The general directed case is weaker: it provides admissible comparisons between neighbouring local state spaces but need not have an inverse. We use the word ``transport'' for both cases and state explicitly whenever reversibility is required.
\end{remark}

\begin{remark}[Relation to quiver representations and cellular sheaves]
The vertex spaces \(F_a\) together with directed edge maps \(P_{ab}\) define a representation of the directed \(1\)-skeleton, viewed as a quiver. They do not automatically define a cellular sheaf, because no stalks on higher-dimensional cells or compatible face-to-coface restriction maps have yet been specified. In the reversible, constant-rank case, path transport may be used to construct local-system data provided that the required path and composition compatibilities are also imposed. The sheaf-theoretic comparison made later in this note is subject to these qualifications.
\end{remark}

\subsection{Path transport}
\label{subsec:Connections_Paths}

Transport along arbitrary paths is obtained by composition.
\begin{definition}[Path transport]
Let $\gamma = (v_0, v_1, \dots, v_k)$ be an oriented discrete path in $\mathcal{K}_0$ from $v_0$ to $v_k.$ The path transport map $P_\gamma: F_{v_0} \to F_{v_k}$ is defined as the ordered composition of elementary neighbour transport maps:
\[
P_\gamma := P_{v_{k-1} v_k} \circ \dots \circ P_{v_1 v_2} \circ P_{v_0 v_1}.
\]
In particular, for a 1-step intermediate path $\gamma = (c, b, a),$ the transport factors explicitly as $P_\gamma = P_{ba} \circ P_{cb}.$
\end{definition}

Path transports permit quantities attached to different vertices to be expressed in a common fibre. They provide the fundamental transport operators from which covariant differentiation, torsion and curvature are constructed. The choice of path becomes part of the geometric data whenever transport is not path independent.

\begin{remark}
If \(P\) is a reversible discrete connection, then every path transport is an isomorphism. For the reversed path
\[
\bar\gamma=(v_k,v_{k-1},\ldots,v_0)
\]
one has
\[
P_{\bar\gamma}=P_\gamma^{-1}.
\]
Invertibility of the individual maps alone is not sufficient for this identity; the reversal condition \(P_{ba}=P_{ab}^{-1}\) is essential.
\end{remark}

\subsection{Dual transport}
\label{subsec:Connections_Dual}

Transport maps between neighbouring fibres induce corresponding transports between dual fibres.

\begin{definition}[Dual transport]
Let
\[
P_{ab} : F_a \longrightarrow F_b
\]
be an edge transport. Its dual transport is the pullback map
\[
P_{ab}^{*} : F_b^{*} \longrightarrow F_a^{*},
\]
defined by
\[
(P_{ab}^{*}\eta)(X) = \eta(P_{ab}X),
\]
for every
\[
\eta\in F_b^{*}, \qquad X\in F_a.
\]
\end{definition}

Thus, vectors are pushed forward in the direction of the transport,
\[
P_{ab} : F_a \rightarrow F_b,
\]
whereas covectors are pulled back,
\[
P_{ab}^{*} : F_b^{*} \rightarrow F_a^{*}.
\]
Every path transport
\[
P_\gamma : F_a \rightarrow F_b
\]
induces a corresponding dual path transport
\[
P_\gamma^{*} : F_b^{*} \rightarrow F_a^{*}
\]
through the same construction. In particular, if a covector attached to the fibre \(F_b^*\) is to be transported to the fibre \(F_a^*\), the required map is
\[
P_{ab}^*:F_b^*\longrightarrow F_a^*,
\]
that is, the pullback associated with the vector transport 
\[
P_{ab}:F_a\longrightarrow F_b.
\]

In the reversible case, the inverse transport induces a canonical map in the same path direction as vector transport. The covector transport associated with
\[
P_{ab}:F_a\longrightarrow F_b
\]
from \(F_a^*\) to \(F_b^*\) is
\[
(P_{ab}^{-1})^*=P_{ba}^*:F_a^*\longrightarrow F_b^*.
\]
For a non-invertible directed transport there is, in general, no canonical covector push-forward \(F_a^*\to F_b^*\); only the pullback \(P_{ab}^*:F_b^*\to F_a^*\) is canonical.

\begin{remark}
The distinction between vector and covector transport is independent of any metric structure. It follows solely from the duality pairing
\[
F_a^{*}\times F_a \rightarrow \mathbb R.
\]
\end{remark}

\subsection{Admissibility of transport structures}
\label{subsec:Connections_Admissibility}

The definition of a directed edge-transport structure specifies only the existence of linear maps between neighbouring fibres. It does not determine these maps uniquely. Consequently, many different transport structures may coexist on the same computational complex.

The choice of an admissible transport structure depends on the intended geometric or physical interpretation. Typical admissibility requirements include locality, invertibility, reversibility, preservation of
distinguished combinatorial directions, metric compatibility, or other algebraic constraints. Different applications may impose different combinations of these conditions. The framework intentionally leaves these choices open.

The present work develops only the general framework. Individual admissibility conditions are introduced progressively as the
corresponding geometric structures become available.

\subsection{Locality constraints}
\label{subsec:Connections_Locality}

One natural admissibility condition concerns the locality of transport. Although the transport maps are defined between the full neighbouring fibres,
\[
P_{ab}:F_a\rightarrow F_b,
\]
their action need not involve every local direction represented in these fibres. For a given edge
\[
e=(a,b),
\]
the fibres
\[
F_a = \operatorname{span}\{\mathbf e_{a,c}:c\in\operatorname{Adj}(a)\},
\qquad
F_b = \operatorname{span}\{\mathbf e_{b,d}:d\in\operatorname{Adj}(b)\},
\]
contain generators associated with all incident 1-cells.

Transport along \(e\), however, may be restricted to directions participating in local
2-cell neighbourhoods containing \(e\). Define
\[
F_a^{(e)}
=
\operatorname{span}
\left\{
\mathbf e_{a,c}
:
\exists\, f\in\mathcal K_2
\text{ such that }
e\prec f
\text{ and }
(a,c)\prec f
\right\}
\subseteq F_a ,
\]
and
\[
F_b^{(e)}
=
\operatorname{span}
\left\{
\mathbf e_{b,d}
:
\exists\, f\in\mathcal K_2
\text{ such that }
e\prec f
\text{ and }
(b,d)\prec f
\right\}
\subseteq F_b .
\]

Thus \(F_a^{(e)}\) and \(F_b^{(e)}\) are generated by the local directions that participate in at least one 2-cell incident to the edge \(e\). The transport associated with \(e\) may then be regarded primarily as a map
\[
P_{ab}:F_a^{(e)}\longrightarrow F_b^{(e)}.
\]

Extensions to the full fibres may be introduced when required.

For a quadrilateral 2-cell, the corresponding local transport spaces contain precisely the edge directions participating in that local 2-cell neighbourhood. Locality is therefore defined combinatorially through incidences rather than through distances or coordinates.

\begin{remark}

Locality constraints do not require transport maps to be invertible or to preserve the full fibre dimension. They restrict only the combinatorial support on which transport acts. Additional admissibility conditions, such as invertibility or metric compatibility, may be imposed independently when required by a particular application.
\end{remark}

\subsection{Transported differences}
\label{subsec:Connections_Differences}

Transport allows neighbouring fibre values to be expressed in a common local state space. Let
\[
X\in C^0(\mathcal{K};F)
\]
be a section and let
\[
e=(a,b)\in \mathcal{K}_1.
\]
The transported value
\[
P_{ab}X(a)
\]
and the local value
\[
X(b)
\]
both belong to the fibre \(F_b\). They can therefore be compared directly. The transported difference
\[
X(b)-P_{ab}X(a)
\]
is the fundamental object from which covariant differential operators are constructed.

Similarly,
\[
X(a)-P_{ba}X(b)
\]
provides the corresponding comparison in the fibre \(F_a\).

\begin{remark}
These transported differences play the role of discrete covariant increments. Their weighted assembly over the computational complex gives rise to the incidence operators introduced later.
\end{remark}

The transport structure provides the comparisons required between neighbouring local states. Once transport is available, one may investigate transported vectors around loops and transported local directions around cell boundaries. These constructions lead respectively to curvature and torsion defects.

\section{Solder Forms and Torsion}
\label{sec:Torsion}

A directed transport structure provides comparisons between neighbouring fibres. A second structure is required to relate the combinatorial structure of the complex to the local coefficient spaces.
This structure is the solder form.

Whereas transport compares fibres attached to different vertices, the solder specifies how the combinatorial structure of the computational complex is represented within each fibre.

\subsection{Solder forms}
\label{subsec:Torsion_Solder}

\begin{definition}[Solder form]
A solder form is a vector-valued 1-cochain
\[
\theta \in C^1(\mathcal{K};F).
\]
Equivalently, for every incidence
\[
(a,e)\in \mathcal{K}_{0,1}, \qquad a\prec e,
\]
the solder assigns a vector
\[
\theta(a,e) \in F_a.
\]
In particular, it satisfies the orientation rule
\[
\theta(a,-e)=-\theta(a,e).
\]
\end{definition}

Thus, the solder associates local fibre representatives with the 1-cells of the computational complex.
It provides the correspondence between combinatorial directions represented by incident edges and directions represented within the local fibres. Unlike transport maps, which compare neighbouring fibres, the solder acts entirely within each fibre individually. 

\paragraph{Canonical solder.}
Let \(e\) be an oriented \(1\)-cell with endpoints \(a\) and \(b\), and write
\[
\partial e=t(e)-s(e).
\]
Thus \(s(e)\) and \(t(e)\) are respectively the initial and terminal vertices of the chosen orientation, and
\[
\varepsilon(e,s(e))=-1,
\qquad
\varepsilon(e,t(e))=+1.
\]
The canonical solder is defined at either endpoint \(a\prec e\) by
\[
\theta_{\mathrm{can}}(a,e)
=-\varepsilon(e,a)\,\mathbf e_{a,b},
\]
where \(b\) is the other endpoint of \(e\). Equivalently,
\[
\theta_{\mathrm{can}}(s(e),e)
=\mathbf e_{s(e),t(e)},
\qquad
\theta_{\mathrm{can}}(t(e),e)
=-\mathbf e_{t(e),s(e)}.
\]
Since \(\varepsilon(-e,a)=-\varepsilon(e,a)\), this definition gives
\[
\theta_{\mathrm{can}}(a,-e)
=-\theta_{\mathrm{can}}(a,e),
\]
as required for a vector-valued \(1\)-cochain. Unless stated otherwise, \(\theta=\theta_{\mathrm{can}}\) will be used.

\begin{remark}[Solder forms as the bridge beyond standard sheaves]
In cellular sheaf theory, stalks are abstract vector spaces without an intrinsic identification with directions in the underlying cell complex \cite{Curry2014sheaves}. The canonical solder anchors the signed fibre basis to the oriented \(1\)-cells incident with each vertex. This additional datum motivates the operational torsion defect defined below.
\end{remark}

\subsection{Endomorphism representation of vector-valued
1-cochains}

The canonical solder establishes a natural correspondence between vector-valued \(1\)-cochains and fibre endomorphisms after one orientation has been fixed for every \(1\)-cell. Let
\[
X\in C^1(\mathcal{K};F).
\]
For every incidence
\[
(a,e)\in \mathcal{K}_{0,1},
\]
the coefficient
\[
X(a,e)\in F_a
\]
belongs to the same fibre as the canonical solder value
\[
\theta(a,e)\in F_a.
\]
Since the canonical solder associates a basis vector of \(F_a\) with every incident \(1\)-cell, the values of \(X\) determine uniquely a linear transformation of the fibre.

\begin{proposition}[Endomorphism representation]
Let the canonical solder be used. For every
\[
X\in C^1(\mathcal{K};F)
\]
and every vertex
\[
a\in \mathcal{K}_0,
\]
there exists a unique endomorphism
\[
A_X(a)\in\operatorname{End}(F_a)
\]
such that
\[
X(a,e)=A_X(a)\,\theta(a,e)
\]
for every incident \(1\)-cell
\[
e\in\operatorname{St}_1(a).
\]
\end{proposition}

\begin{proof}
For the chosen edge orientations, the vectors
\[
\{\theta(a,e):e\in\operatorname{St}_1(a)\}
\]
are signed versions of the incidence basis and hence form a basis of \(F_a\). A linear map is uniquely determined by its action on a basis. The values
\[
X(a,e)\in F_a
\]
therefore prescribe the image of every basis vector and hence determine a unique endomorphism
\[
A_X(a).
\]
\end{proof}

Consequently, every vector-valued \(1\)-cochain determines a unique endomorphism-valued \(0\)-cochain
\[
A_X\in C^0(\mathcal{K};\operatorname{End}(F)).
\]
Equivalently,
\[
X(a,e)=A_X(a)\,\theta(a,e)
\]
for every incidence pair
\[
(a,e)\in \mathcal{K}_{0,1}.
\]
The trace
\[
\operatorname{tr}(A_X(a))
\]
is defined independently of any fibre metric and is invariant under changes of basis of \(F_a\). Thus, every vector-valued \(1\)-cochain possesses an associated metric-independent scalar field
\[
a\longmapsto \operatorname{tr}(A_X(a)).
\]

\paragraph{Example.}
Suppose
\[
\dim(F_a)=3
\]
and the vertex \(a\) is incident with the \(1\)-cells
\[
e_1,e_2,e_3.
\]
The canonical solder basis is
\[
\theta(a,e_1)=\mathbf e_1,\qquad
\theta(a,e_2)=\mathbf e_2,\qquad
\theta(a,e_3)=\mathbf e_3.
\]
Let
\[
X(a,e_1)=v_1,\qquad
X(a,e_2)=v_2,\qquad
X(a,e_3)=v_3,
\]
where
\[
v_i\in F_a.
\]
Expanding these vectors in the solder basis,
\[
v_1=
\begin{bmatrix}
v_{11}\\
v_{21}\\
v_{31}
\end{bmatrix},
\qquad
v_2=
\begin{bmatrix}
v_{12}\\
v_{22}\\
v_{32}
\end{bmatrix},
\qquad
v_3=
\begin{bmatrix}
v_{13}\\
v_{23}\\
v_{33}
\end{bmatrix},
\]
the corresponding endomorphism is defined by
\[
A_X(a)\mathbf e_i=v_i.
\]
Its matrix in the solder basis is therefore
\[
A_X(a) = \begin{bmatrix}
v_{11} & v_{12} & v_{13}\\
v_{21} & v_{22} & v_{23}\\
v_{31} & v_{32} & v_{33}
\end{bmatrix}.
\]
The columns of the matrix are precisely the vectors assigned by the \(1\)-cochain to the three incident \(1\)-cells.

\begin{remark}[Dependence on the canonical solder]
Proposition~5.3 relies on the fact that the canonical solder identifies each incident
1-cell with a basis vector of the corresponding fibre. For a general solder form,
the values
\[
\{\theta(a,e):e\in\operatorname{St}_1(a)\}
\]
need not form a basis of \(F_a\).
In that case a vector-valued \(1\)-cochain need not determine a unique
endomorphism-valued \(0\)-cochain.
The endomorphism representation is therefore a special property of the canonical
solder used throughout this note.
\end{remark}

\subsection{Transported solder}
\label{subsec:Torsion_Transported}

The vectors assigned by the solder belong to different fibres and cannot be added directly. Directed transport maps permit them to be expressed in a common reference fibre.

Let
\[
f\in \mathcal{K}_2
\]
be an oriented quadrilateral and let
\[
a\prec f
\]
be a chosen reference vertex. Fix one orientation for every boundary edge
\[
e\prec f,
\]
and let \(s(e)\) denote its initial vertex. For every triple \((a,f,e)\) with \(a\prec f\) and \(e\prec f\), choose a path
\[
\gamma^a_{f,e}:s(e)\rightsquigarrow a
\]
in the \(1\)-skeleton. The choice is required to depend on the underlying unoriented \(2\)-cell, the reference vertex, and the oriented edge representative, but not on the choice between \(f\) and \(-f\). If \(s(e)=a\), the constant path is permitted and its transport is the identity. The transported solder vector
\[
P_{\gamma^a_{f,e}}\theta(s(e),e)\in F_a
\]
then represents the globally oriented edge \(e\) in the reference fibre.

\begin{figure}[h!]
\centering
\begin{tikzpicture}[scale=1.1,>=Latex]
\coordinate (A1) at (0,0);
\coordinate (A2) at (5,0);
\coordinate (A3) at (5,3);
\coordinate (A4) at (0,3);

\draw[->,thick] (A1)--(A2) node[midway,below] {\(e_1\)};
\draw[->,thick] (A2)--(A3) node[midway,right] {\(e_2\)};
\draw[->,thick] (A3)--(A4) node[midway,above] {\(e_3\)};
\draw[->,thick] (A4)--(A1) node[midway,left] {\(e_4\)};

\fill (A1) circle (2.5pt) node[below left] {\(a_1=a\)};
\fill (A2) circle (2.5pt) node[below right] {\(a_2\)};
\fill (A3) circle (2.5pt) node[above right] {\(a_3\)};
\fill (A4) circle (2.5pt) node[above left] {\(a_4\)};
\node at (2.5,1.5) {\(f\)};
\end{tikzpicture}
\caption{A compatible choice of orientations for the canonical quadrilateral.}
\label{fig:canonical-quadrilateral}
\end{figure}
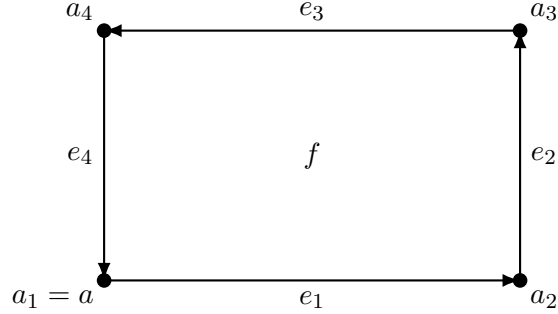

\subsection{Torsion}
\label{subsec:Torsion_Definition}

\begin{definition}[Torsion]
\label{def:torsion}
Fix oriented representatives of the boundary \(1\)-cells and a path system \(\{\gamma^a_{f,e}\}\) as above. The transported-solder defect associated with the incidence pair
\[
(a,f)
\]
is
\[
\Theta(a,f)
=\sum_{e\prec f}
\varepsilon(f,e)\,
P_{\gamma^a_{f,e}}\theta(s(e),e).
\]
\end{definition}

Since every term belongs to the fibre \(F_a\),
\[
\Theta(a,f) \in F_a.
\]
Thus, torsion is naturally a vector-valued 2-cochain,
\[
\Theta \in C^2(\mathcal{K};F).
\]

\begin{proposition}[Orientation covariance]
For every incidence \(a\prec f\),
\[
\Theta(a,-f)=-\Theta(a,f).
\]
\end{proposition}

\begin{proof}
The edge representatives, their source vertices, and the path system are unchanged when the orientation of \(f\) is reversed, whereas
\[
\varepsilon(-f,e)=-\varepsilon(f,e).
\]
The defining sum therefore changes sign term by term.
\end{proof}

\begin{remark}[Dependence on the path system]
For a general directed transport structure, \(\Theta(a,f)\) depends on the chosen paths \(\gamma^a_{f,e}\). This dependence is part of the present operational definition. Path independence requires additional assumptions, such as an appropriate flatness condition on the region containing the competing paths.
\end{remark}

\subsection{Canonical quadrilateral}
\label{subsec:Torsion_Example}

\noindent
Let \(f\) have the boundary orientation shown in
Figure~\ref{fig:canonical-quadrilateral}, so that
\[
e_1=(a_1,a_2),\qquad
e_2=(a_2,a_3),\qquad
e_3=(a_3,a_4),\qquad
e_4=(a_4,a_1),
\]
and \(\varepsilon(f,e_i)=+1\) for \(i=1,\ldots,4\). Choose \(a_1\) as the reference vertex and the paths
\[
\gamma^{a_1}_{f,e_1}=(a_1),\qquad
\gamma^{a_1}_{f,e_2}=(a_2,a_1),\qquad
\gamma^{a_1}_{f,e_3}=(a_3,a_4,a_1),\qquad
\gamma^{a_1}_{f,e_4}=(a_4,a_1).
\]
Writing \(P_{ij}:F_{a_i}\to F_{a_j}\), the transported-solder defect is
\[
\begin{aligned}
\Theta(a_1,f)={}&
\theta(a_1,e_1)
+P_{21}\theta(a_2,e_2)\\
&+P_{41}P_{34}\theta(a_3,e_3)
+P_{41}\theta(a_4,e_4).
\end{aligned}
\]
For the canonical solder,
\[
\begin{aligned}
\theta(a_1,e_1)&=\mathbf e_{a_1,a_2},&
\theta(a_2,e_2)&=\mathbf e_{a_2,a_3},\\
\theta(a_3,e_3)&=\mathbf e_{a_3,a_4},&
\theta(a_4,e_4)&=\mathbf e_{a_4,a_1},
\end{aligned}
\]
and hence
\[
\begin{aligned}
\Theta(a_1,f)={}&
\mathbf e_{a_1,a_2}
+P_{21}\mathbf e_{a_2,a_3}\\
&+P_{41}P_{34}\mathbf e_{a_3,a_4}
+P_{41}\mathbf e_{a_4,a_1}.
\end{aligned}
\]
For arbitrary choices of the global edge orientations, the incidence factors \(\varepsilon(f,e_i)\) in Definition~\ref{def:torsion} restore the boundary signs.

\subsection{Interpretation}
\label{subsec:Torsion_Interpretation}

The transported-solder defect measures the failure of the selected and transported local edge directions to close around a \(2\)-cell. When
\[
\Theta(a,f)=0,
\]
the transported boundary directions form a closed loop in the reference fibre. When
\[
\Theta(a,f)\neq 0,
\]
their closure defect is represented by \(\Theta(a,f)\).

For the canonical solder, this quantity measures incompatibility between the incidence directions and the chosen transport structure. We use the term \emph{torsion} for this operational defect, while leaving its precise relation to a future covariant Cartan equation as an open problem.

\begin{remark}[Incidence attachment versus tensoriality]
The result is a vector attached to a single incidence,
\[
\Theta(a,f)\in F_a.
\]
This makes it a vector-valued \(2\)-cochain under the prescribed orientation and path conventions. It does not, by itself, prove the stronger tensorial or gauge-covariance properties of smooth torsion. Such properties require transformation laws for the fibre bases, solder form, and transport maps, together with a proof that the resulting defect transforms covariantly.
\end{remark}

\section{Curvature}
\label{sec:Curvature}

Whereas the torsion defect depends on both the solder form and the transport structure, holonomy depends only on transport. At finite scale, however, the complete holonomy defect and the orientation-sensitive part of holonomy are distinct quantities. This section introduces both objects and distinguishes their roles explicitly.

\subsection{Holonomy}
\label{subsec:Curvature_Holonomy}

Let \(f\in\mathcal{K}_2\) be an oriented quadrilateral and let \(a\prec f\) be a reference vertex. Starting at \(a\), traverse the oriented boundary of \(f\) once and compose the directed edge transports in traversal order.

\begin{definition}[Holonomy]
\label{def:holonomy}
The holonomy associated with the incidence pair
\[
(a,f)
\]
is the endomorphism
\[
\mathcal{R}(a,f): F_a \rightarrow F_a
\]
obtained by composing the transport maps along the oriented boundary of \(f\).
\end{definition}

For the oriented quadrilateral
\[
a_1 \xrightarrow{e_1} a_2 \xrightarrow{e_2} a_3 \xrightarrow{e_3} a_4 \xrightarrow{e_4} a_1,
\]
choosing \(a_1\) as the reference vertex gives
\[
\mathcal{R}(a_1,f) = P_{41} P_{34} P_{23} P_{12}.
\]
The reversed orientation gives
\[
\mathcal{R}(a_1,-f) = P_{21}P_{32}P_{43}P_{14}.
\]
For a general directed transport structure, these two endomorphisms are independent. If the transport is reversible along the boundary, then
\[
\mathcal{R}(a,-f)=\mathcal{R}(a,f)^{-1}.
\]

\begin{remark}[Dependence on the reference vertex]
For a general directed transport structure, the holonomy
\[
\mathcal R(a,f)
\]
depends on the chosen reference vertex \(a\prec f\).
If the transport is reversible, the holonomies associated with different
reference vertices are related by conjugation through the transport along
the boundary between the corresponding vertices. Consequently, conjugacy
invariants such as the trace are independent of the chosen reference
vertex.
\end{remark}

\subsection{Finite holonomy defect and curvature}
\label{subsec:Curvature_Definition}

\begin{definition}[Finite holonomy defect]
\label{def:finite-holonomy-defect}
The finite holonomy defect associated with \((a,f)\) is
\[
\Delta(a,f) = \mathcal{R}(a,f)-\operatorname{Id}_{F_a}.
\]
\end{definition}

The defect \(\Delta(a,f)\) measures the full discrepancy between a vector and its transport around the oriented boundary. In general,
\[
\Delta(a,-f)\ne-\Delta(a,f),
\]
even for reversible transport. It must therefore not be identified with an endomorphism-valued \(2\)-cochain under the orientation convention of Section~\ref{sec:BundleValued}.

\begin{definition}[Orientation-odd curvature]
\label{def:orientation-odd-curvature}
The orientation-odd curvature is the endomorphism
\[ \Omega(a,f) = \frac12\left( \mathcal{R}(a,f)-\mathcal{R}(a,-f) \right).
\]
\end{definition}

\begin{proposition}[Orientation covariance]\label{prop:curvature-orientation}
The orientation-odd curvature satisfies
\[
\Omega(a,-f)=-\Omega(a,f).
\]
\end{proposition}

\begin{proof}
Interchanging \(f\) and \(-f\) reverses the two terms in Definition~\ref{def:orientation-odd-curvature}.
\end{proof}

Consequently,
\[
\Omega\in C^2(\mathcal{K};\operatorname{End}(F)).
\]
For reversible transport,
\[ \Omega(a,f) = \frac12\left( \mathcal{R}(a,f)-\mathcal{R}(a,f)^{-1} \right).
\]
If \(\mathcal{R}(a,f)\) is close to the identity, then \(\Omega(a,f)\) and \(\Delta(a,f)\) agree to first order, but they are distinct finite-scale quantities.

\subsection{Change of fibre bases and scalar invariants}

Let \(G_a:F_a\to F_a\) be an invertible change of basis at each vertex. The corresponding transformed edge maps are
\[
P^{G}_{ab}=G_bP_{ab}G_a^{-1}.
\]

\begin{proposition}[Gauge covariance of holonomy and curvature] \label{prop:curvature-gauge}
Under these changes of fibre bases,
\[
\mathcal{R}^{G}(a,f) = G_a\mathcal{R}(a,f)G_a^{-1},
\]
\[
\Delta^{G}(a,f) = G_a\Delta(a,f)G_a^{-1},
\]
\[
\Omega^{G}(a,f) = G_a\Omega(a,f)G_a^{-1}.
\]
\end{proposition}

\begin{proof}
In the composition around a closed boundary, the factors \(G_v^{-1}G_v\) cancel at every intermediate vertex. Only the factors at the reference vertex remain. The formulas for \(\Delta^G\) and \(\Omega^G\) follow from their definitions.
\end{proof}

It follows that conjugacy invariants of \(\mathcal{R}\), \(\Delta\), and \(\Omega\) are independent of the chosen fibre bases. In particular,
\[
(a,f) \longmapsto \operatorname{tr}\left(\Omega(a,f)\right),
\]
is a metric-independent, orientation-odd scalar incidence function. It should not be confused with an ordinary scalar cochain in \(C^2(\mathcal{K})\), which assigns only one value to each \(2\)-cell rather than one value to every vertex--cell incidence. The trace of \(\Delta(a,f)\) is also basis independent, but it is not generally alternating in \(f\).

\subsection{Flat transport}\label{subsec:Curvature_Flat}

\begin{definition}[Two-sided face flatness]
The directed transport structure is flat on the incidence \((a,f)\) if
\[
\mathcal{R}(a,f) = \operatorname{Id}_{F_a},
\qquad
\mathcal{R}(a,-f)=\operatorname{Id}_{F_a}.
\]
\end{definition}

Equivalently,
\[
\Delta(a,f)=\Delta(a,-f)=0.
\]
Two-sided face flatness implies \(\Omega(a,f)=0\), but the converse is not true: \(\Omega(a,f)=0\) only states that the two oriented holonomies agree. In the reversible case, it states that
\[
\mathcal{R}(a,f)=\mathcal{R}(a,f)^{-1},
\]
which also permits non-identity involutions.

\begin{remark}[Local flatness and local systems]
For a reversible constant-rank transport structure, identity holonomy on every contractible loop is the condition expected of flat local transport. Identity holonomy around the boundary of each \(2\)-cell is sufficient for path independence only under additional topological hypotheses, for example when every loop under consideration is generated by such cell boundaries. Non-contractible loops may retain non-trivial monodromy. Consequently, even face flatness does not by itself imply globally path-independent transport.
\end{remark}

\subsection{Interpretation}
\label{subsec:Curvature_Interpretation}

Holonomy compares a vector with the result obtained after transporting it around a closed boundary. The complete finite discrepancy is represented by
\[
\Delta(a,f) = \mathcal{R}(a,f)-\operatorname{Id}_{F_a}.
\]
It vanishes precisely when transport around that oriented boundary leaves every vector unchanged. By contrast,
\[
\Omega(a,f) = \frac12\left(\mathcal{R}(a,f)-\mathcal{R}(a,-f)\right)
\]
extracts the part that changes sign with the orientation and therefore has the algebraic type of an endomorphism-valued \(2\)-cochain.

Thus the three operational quantities capture different aspects of the
geometry:
\begin{itemize}
\item \(\Theta\) measures transported closure defects of soldered local
directions and depends on the solder form, the path system, and the
transport maps;
\item \(\Delta\) measures the complete finite failure of transport to return
vectors unchanged after a circuit of the boundary;
\item \(\Omega\) measures the orientation-sensitive part of the holonomy and
therefore has the algebraic type of an endomorphism-valued \(2\)-cochain.
\end{itemize}

\begin{remark}[Curvature as localised transport information]
Although \(\Delta(a,f)\) and \(\Omega(a,f)\) are attached to a single incidence, both depend on transport around the entire boundary of \(f\). Their incidence attachment should not be confused with dependence on data at the reference vertex alone. The covariance statement in Proposition \ref{prop:curvature-gauge} gives the precise sense in which they are local fibre endomorphisms.
\end{remark}

\begin{remark}[Terminology]
In the remainder of this note, \emph{finite holonomy defect} refers to
\[
\Delta(a,f)=\mathcal{R}(a,f)-\operatorname{Id}_{F_a},
\]
whereas \emph{curvature cochain} refers to the alternating quantity \(\Omega\). Statements concerning exact return of transported vectors must be formulated using \(\Delta\), not merely \(\Omega\).
\end{remark}

\section{Weighted Covariant Incidence Operators}
\label{sec:Covariant}

The ordinary coboundary operator acts on scalar-valued cochains because all coefficients belong to the same vector space and each boundary contribution can be summed directly. For incidence-based fibre coefficients, contributions associated with different vertices lie in different fibres and must first be transported to a common reference fibre. Moreover, a single boundary cell carries one coefficient at each of its incident vertices, so an additional rule is needed to avoid unintended multiple counting. The construction below therefore requires transport paths, anchor sets, and normalising weights.

\subsection{Weighted covariant incidence operator}
\label{subsec:Covariant_Operator}

\paragraph{Anchoring data.}
For a cell \(\rho\), write
\[
\operatorname{Vert}(\rho) = \{\,b\in\mathcal{K}_0:b\prec\rho\,\}.
\]
Fix a degree \(p\). For every incidence
\[
(a,\tau)\in\mathcal{K}_{0,p+1},
\]
choose the following data:
\begin{enumerate}
\item an anchor set
\[
S_p(a,\tau)\subseteq\operatorname{Vert}(\tau), \qquad a\in S_p(a,\tau),
\]
such that
\[
S_p(a,\tau)\cap\operatorname{Vert}(\sigma)\ne\varnothing
\]
for every \(p\)-cell \(\sigma\prec\tau\);
\item for each \(b\in S_p(a,\tau)\), a path
\[
\gamma^a_{\tau,b}:b\rightsquigarrow a
\]
in the \(1\)-skeleton, with the constant path used when \(b=a\);
\item weights
\[
w^p_{a,\tau}(b,\sigma)\in\mathbb{R},
\qquad
b\in S_p(a,\tau)\cap\operatorname{Vert}(\sigma),
\]
satisfying the partition condition
\[
\sum_{b\in S_p(a,\tau)\cap\operatorname{Vert}(\sigma)} w^p_{a,\tau}(b,\sigma)=1
\]
for every \(\sigma\prec\tau\).
\end{enumerate}
The anchor sets, paths, and weights are required to be independent of the orientation of \(\tau\):
\[
S_p(a,-\tau)=S_p(a,\tau),\qquad\gamma^a_{-\tau,b}=\gamma^a_{\tau,b},\qquad w^p_{a,-\tau}(b,\sigma)=w^p_{a,\tau}(b,\sigma).
\]
We denote the complete collection of this degree-\(p\) data by
\[
\mathfrak{A}_p=(S_p,\Gamma_p,w^p).
\]

\paragraph{Representative choices.}
The full-vertex choice
\[
S_p(a,\tau)=\operatorname{Vert}(\tau)
\]
is always admissible. A canonical normalisation for any admissible anchor set is
\[
w^p_{a,\tau}(b,\sigma) = \frac{1}{ \#\bigl(S_p(a,\tau)\cap\operatorname{Vert}(\sigma)\bigr)}.
\]
Smaller anchor sets and non-uniform weights define different locality prescriptions. They are part of the operator and must therefore be reported whenever \(\delta^\nabla\) is used.

\begin{definition}[Weighted covariant incidence operator]
\label{def:covariant-incidence-operator}
Let \(X\in C^p(\mathcal{K};F)\). Relative to anchoring data \(\mathfrak{A}_p\), define
\[
\begin{aligned}
(\delta^\nabla_{\mathfrak{A}_p}X)(a,\tau) ={}&
\sum_{\sigma\prec\tau}\varepsilon(\tau,\sigma)
\sum_{\substack{b\in S_p(a,\tau)\\ b\prec\sigma}}
w^p_{a,\tau}(b,\sigma)\,
P_{\gamma^a_{\tau,b}}X(b,\sigma).
\end{aligned}
\]
\end{definition}

Every summand belongs to \(F_a\), so the expression is well defined. The notation \(\delta^\nabla\) may be used when \(\mathfrak{A}_p\) has been fixed and no ambiguity can arise.

\begin{proposition}[Linearity and orientation covariance]
\label{prop:covariant-linearity-orientation}
For fixed \(P\) and \(\mathfrak{A}_p\), the assignment in Definition~\ref{def:covariant-incidence-operator} is a linear map
\[
\delta^\nabla_{\mathfrak{A}_p}: C^p(\mathcal{K};F) \longrightarrow C^{p+1}(\mathcal{K};F).
\]
In particular,
\[
(\delta^\nabla_{\mathfrak{A}_p}X)(a,-\tau) = -(\delta^\nabla_{\mathfrak{A}_p}X)(a,\tau).
\]
\end{proposition}

\begin{proof}
Linearity follows from the linearity of the transport maps and the finite sums. Under \(\tau\mapsto-\tau\), the anchoring data remain unchanged while
\[
\varepsilon(-\tau,\sigma)=-\varepsilon(\tau,\sigma).
\]
The defining expression therefore changes sign.
\end{proof}

\begin{proposition}[Covariance under fibre-basis changes]
\label{prop:covariant-gauge}
Let \(G_a:F_a\to F_a\) be invertible and set
\[
X^G(a,\sigma)=G_aX(a,\sigma),
\qquad
P^G_{ab}=G_bP_{ab}G_a^{-1}.
\]
Then
\[
\bigl(\delta^{\nabla,G}_{\mathfrak{A}_p}X^G\bigr)(a,\tau)
=G_a
\bigl(\delta^\nabla_{\mathfrak{A}_p}X\bigr)(a,\tau).
\]
\end{proposition}

\begin{proof}
For every anchor path from \(b\) to \(a\),
\[
P^G_{\gamma^a_{\tau,b}}
=G_aP_{\gamma^a_{\tau,b}}G_b^{-1}.
\]
Substitution into Definition~\ref{def:covariant-incidence-operator} cancels \(G_b^{-1}G_b\) in every summand and leaves the common factor \(G_a\).
\end{proof}

\begin{remark}[The case \(p=0\)]
Let
\[
X\in C^0(\mathcal{K};F)
\]
and let
\[
e=(a,b)\in \mathcal{K}_1.
\]
Choose \(S_0(a,e)=\{a,b\}\), the constant path at \(a\), the edge path from \(b\) to \(a\), and the necessarily unit weights on the two \(0\)-faces. Since
\[
\varepsilon(e,a)=-1,\qquad
\varepsilon(e,b)=+1,
\]
the operator reduces to
\[
(\delta^\nabla_{\mathfrak{A}_0}X)(a,e)
=P_{ba}X(b)-X(a).
\]
Thus, the elementary transported difference introduced in Section~\ref{subsec:Connections_Differences} is recovered as the lowest-order instance of the covariant coboundary.
\end{remark}

\paragraph{Extension to other coefficient families.}

For a path \(\gamma:b\rightsquigarrow a\), the vector transport
\[
P_\gamma:F_b\longrightarrow F_a
\]
does not canonically transport covectors from \(F_b^*\) to \(F_a^*\). The canonical dual map is the pullback
\[
P_\gamma^*:F_a^*\longrightarrow F_b^*,
\]
which has the opposite direction on dual fibres. Therefore, to transport a covector from \(F_b^*\) to \(F_a^*\), one needs the pullback associated with a reverse directed path
\[
\bar\gamma:a\rightsquigarrow b.
\]
When such a reverse path transport is available, the required map is
\[
P_{\bar\gamma}^*:F_b^*\longrightarrow F_a^*.
\]

Thus, for
\[
\eta\in C^p(\mathcal{K};F^*),
\]
and for anchoring data whose reverse paths are admissible for covector transport, define
\[
\begin{aligned}
(\delta^{\nabla,*}_{\mathfrak{A}_p}\eta)(a,\tau) ={}&
\sum_{\sigma\prec\tau}\varepsilon(\tau,\sigma)
\sum_{\substack{b\in S_p(a,\tau)\\ b\prec\sigma}}
w^p_{a,\tau}(b,\sigma)\,
P_{\overline{\gamma^a_{\tau,b}}}^{*}\eta(b,\sigma).
\end{aligned}
\]
This gives a linear map
\[
\delta^{\nabla,*}_{\mathfrak{A}_p}:
C^p(\mathcal{K};F^*)
\longrightarrow
C^{p+1}(\mathcal{K};F^*).
\]
For reversible transport,
\[
P_{\overline{\gamma}}=P_\gamma^{-1},
\qquad
P_{\overline{\gamma}}^*=(P_\gamma^{-1})^*.
\]

For a non-reversible directed transport structure, the covector operator is therefore not determined by the vector operator alone; it additionally requires admissible reverse path transports or an independently specified dual-transport rule.

The extension of the covariant coboundary to endomorphism-valued cochains,
\[
C^p(\mathcal{K};\operatorname{End}(F)),
\]
requires a natural notion of transport for fibre endomorphisms. When all transport maps are invertible, this is provided by conjugation,
\[
A \longmapsto P_\gamma\circ A\circ P_\gamma^{-1}.
\]
More general constructions for non-invertible transports are left for future investigation.

\begin{remark}[Operator rather than cohomological coboundary]
The name ``covariant coboundary'' describes the incidence and degree-raising form of the construction. It does not assert
\[
\delta^\nabla_{\mathfrak{A}_{p+1}}
\delta^\nabla_{\mathfrak{A}_p}=0.
\]
Without such a nilpotency result, the sequence is not a cochain complex and \(\delta^\nabla\) is more precisely a weighted covariant incidence operator.
\end{remark}

\begin{remark}[Comparison with sheaf coboundaries]
In cellular sheaf cohomology \cite{Shepard1985cellular,Curry2014sheaves}, the differential
\[
d_\mathcal{F}:C^p(\mathcal{K};\mathcal{F})
\longrightarrow C^{p+1}(\mathcal{K};\mathcal{F})
\]
is assembled from compatible restriction maps belonging to a cellular sheaf. The operator \(\delta^\nabla_{\mathfrak A_p}\) defined here is analogous in that it compares coefficients through linear maps before summation. It is not a sheaf coboundary: it acts on vertex--cell incidence data and depends on prescribed anchors, weights, and paths that need not satisfy the functorial axioms of a cellular sheaf.
\end{remark}

\subsection{Status of an exterior covariant derivative}
\label{subsec:Covariant_Exterior}

The operator \(\delta^\nabla_{\mathfrak{A}_p}\) is presently defined only on the incidence-based spaces \(C^p(\mathcal{K};F)\). A space
\[
\Omega^p(\mathcal{M};F)
\]
of fibre-valued Forman differential forms, together with a correspondence to these incidence cochains, has not been constructed in this note. It would therefore be premature to define an operator
\[
d^\nabla:\Omega^p(\mathcal{M};F)
\longrightarrow\Omega^{p+1}(\mathcal{M};F)
\]
by transport of structure.

\begin{definition}[Incidence exterior covariant operator]
For fixed anchoring data \(\mathfrak{A}_p\), the degree-raising map
\[
\delta^\nabla_{\mathfrak{A}_p}:
C^p(\mathcal{K};F)
\longrightarrow C^{p+1}(\mathcal{K};F)
\]
is called the incidence exterior covariant operator.
\end{definition}

The notation \(d^\nabla\) is reserved for a future operator on a rigorously defined space of fibre-valued combinatorial differential forms. Establishing the required correspondence with incidence cochains is one of the open problems of the framework.

\subsection{Relation to the scalar theory}
\label{subsec:Covariant_Scalar}

It is useful to distinguish the present construction from a coefficient-wise vector-valued version of scalar CMC. The incidence-based framework developed here is not obtained by replacing the scalar coefficients of ordinary cochains by vectors or covectors. Instead, it introduces a different class of algebraic objects.

In scalar CMC, a \(p\)-cochain assigns a single coefficient to every \(p\)-cell,
\[
\alpha(\sigma), \qquad \sigma\in\mathcal{K}_p.
\]
By contrast, an fibre-valued incidence \(p\)-cochain assigns one coefficient to every incident pair
\[
(a,\sigma), \qquad a\prec\sigma,
\]
so that the coefficient space depends simultaneously on the supporting cell and on the boundary vertex at which it is evaluated. Consequently, scalar-valued and fibre-valued incidence cochains are defined on different underlying cochain spaces.

The covariant incidence operator is therefore not obtained simply by replacing the scalar coefficients of the ordinary coboundary with vectors. It acts on a larger incidence-based space and includes explicit aggregation weights and transport paths.

No fibre-valued Forman differential-form space or direct extension of Forman's scalar correspondence is asserted here. The present construction remains on the incidence cochains of the computational complex.

There is nevertheless a useful consistency test. Define the vertex-independent lift of a scalar \(p\)-cochain \(\alpha\in C^p(\mathcal{K})\) by
\[
(\iota\alpha)(b,\sigma)=\alpha(\sigma)
\]
for every \(b\prec\sigma\), taking every fibre to be \(\mathbb{R}\).

\begin{proposition}[Recovery of the scalar coboundary]\label{prop:scalar-recovery}
Suppose \(F_a=\mathbb{R}\) for every vertex, all path transports are the identity, and the weights satisfy the partition condition in the definition of \(\mathfrak{A}_p\). Then
\[
(\delta^\nabla_{\mathfrak{A}_p}\iota\alpha)(a,\tau) = (\delta\alpha)(\tau)
\]
for every \(a\prec\tau\).
\end{proposition}

\begin{proof}
Substitution gives
\[
\begin{aligned}
(\delta^\nabla_{\mathfrak{A}_p}\iota\alpha)(a,\tau) &= \sum_{\sigma\prec\tau}\varepsilon(\tau,\sigma)
\alpha(\sigma) \sum_{\substack{b\in S_p(a,\tau)\\b\prec\sigma}} w^p_{a,\tau}(b,\sigma)\\
&=
\sum_{\sigma\prec\tau}
\varepsilon(\tau,\sigma)\alpha(\sigma)
=(\delta\alpha)(\tau).
\end{aligned}
\]
\end{proof}

Thus the normalisation of the anchor weights is essential: without it, the scalar boundary contributions are generally counted more than once.

\begin{remark}[Non-nilpotency and Curvature Obstruction]
Unlike the scalar coboundary, the family of operators
\[
\delta^\nabla_{\mathfrak{A}_{p+1}}
\delta^\nabla_{\mathfrak{A}_p}
\]
is not asserted to vanish. Its composition depends on the anchoring data in two successive degrees as well as on the transport maps. Finite holonomy data should enter any useful decomposition of this composition, but no such identity is proved here. In particular, the smooth formula \(d^\nabla d^\nabla X=\Omega\wedge X\) cannot be transferred merely by analogy: an action product between endomorphism-valued and vector-valued incidence cochains must first be defined, and the weights and path conventions must be shown to be compatible.
\end{remark}

\subsection{Towards Cartan structure equations}
\label{subsec:Covariant_Cartan}

The torsion and curvature introduced in Sections~\ref{sec:Torsion} and~\ref{sec:Curvature} were defined operationally through transported solder forms and holonomy. The weighted incidence operator provides a partial connection between the transported-solder construction and a degree-raising operator.

\begin{remark}[Operations before equations]
The logical order adopted here differs from that of the smooth theory. Torsion and curvature are introduced first through finite transport operations on an explicit cell complex. Differential expressions are sought only afterwards. This reflects the primary role played by transport maps and transported comparisons in the combinatorial setting.
\end{remark}

\begin{proposition}[Recovery of the transported-solder defect]
\label{prop:torsion-recovery}
Let \(p=1\) and choose \(S_1(a,f)\) to contain the source \(s(e)\) of every boundary edge, for example
\[
S_1(a,f)=\operatorname{Vert}(f).
\]
For each \(e\prec f\), choose the anchor weights
\[
w^1_{a,f}(b,e) = \begin{cases}
1,&b=s(e),\\
0,&b\ne s(e),
\end{cases}
\]
and identify the incidence-operator paths with the torsion paths by
\[
\gamma^a_{f,s(e)}=\gamma^a_{f,e}.
\]
Then
\[
(\delta^\nabla_{\mathfrak{A}_1}\theta)(a,f) = \Theta(a,f).
\]
\end{proposition}

\begin{proof}
For each boundary edge, the inner weighted sum in Definition~\ref{def:covariant-incidence-operator} retains only the term anchored at \(s(e)\). Hence
\[
\begin{aligned}
(\delta^\nabla_{\mathfrak{A}_1}\theta)(a,f)
&=
\sum_{e\prec f}
\varepsilon(f,e)
P_{\gamma^a_{f,s(e)}}\theta(s(e),e)\\
&=\Theta(a,f).
\end{aligned}
\]
\end{proof}

\begin{remark}
Proposition~\ref{prop:torsion-recovery} is an equality for a specific anchoring prescription. It does not yet constitute a full Cartan structure equation: such an equation would additionally require an independently defined connection-valued cochain, an action product, and transformation laws relating all terms.
\end{remark}

\begin{problem}
Can the finite holonomy defect or the orientation-odd curvature be recovered from a rigorously defined covariant derivative of connection data?
\end{problem}

\begin{problem}
Under what assumptions are the operational and differential descriptions equivalent?
\end{problem}

Answers to these questions would provide combinatorial analogues of Cartan's structure equations and clarify the relationship between finite transport maps and covariant differential operators on explicit cellular fabrics.

\section{Metric Structures}
\label{sec:Metrics}

The fibres, solder forms, and transport maps introduced previously are independent of any notion of length, angle, or orthogonality. Metric information enters through inner products on the local fibres. A fibre metric is independent data in general; particular choices may be constructed from the cell measures used in scalar Combinatorial Mesh Calculus \cite{BERBATOV2026}, together with additional pairwise information between incident directions.

Throughout this section, the transport structure and the metric are treated as independent data. Compatibility conditions are imposed only when required.

\subsection{Fibre metrics}
\label{subsec:Metrics_Local}

Let
\[
a\in \mathcal{K}_0
\]
and let
\[
F_a = \operatorname{span} \{\mathbf e_{a,b}:
b\in\operatorname{Adj}(a)\}.
\]

\begin{definition}[Fibre metric]
A fibre metric at \(a\) is a symmetric positive-definite bilinear form
\[
g_a : F_a\times F_a \rightarrow \mathbb R.
\]
The collection
\[
g = \{g_a\}_{a\in \mathcal{K}_0}
\]
is called a combinatorial fibre metric.
\end{definition}

\paragraph{Representations in the fibre basis.}

The definition of a fibre metric does not prescribe its matrix with respect to the incident-edge basis
\[
\{\mathbf e_{a,b}:b\in\operatorname{Adj}(a)\}.
\]
Different choices of this Gram matrix correspond to different metric structures on the same fibre. Two representative examples are considered below.

For \(b\in\operatorname{Adj}(a)\), let \(e_{ab}\) denote the unique \(1\)-cell joining \(a\) to \(b\), and set
\[
\ell_{ab}=\mu(e_{ab})>0.
\]

\paragraph{Example: diagonal metric.}

The simplest choice assumes that the basis generators associated with incident 1-cells are mutually orthogonal. For
\[
b_i,b_j\in\operatorname{Adj}(a),
\]
define
\[
g_a(\mathbf e_{a,b_i},\mathbf e_{a,b_j}) = \begin{cases} \ell_{ab_i}^{\,2}, & i=j,\\[2mm] 0, & i\neq j. \end{cases}
\]
Thus, the norm of a basis generator is determined directly by the measure of the corresponding 1-cell.

\paragraph{Example: coupled metric.}

A more general metric permits non-zero pairings between distinct incident directions. Let
\[
\rho_a(b_i,b_j)
\]
be a dimensionless symmetric coefficient satisfying
\[
\rho_a(b_i,b_j)=\rho_a(b_j,b_i), \qquad
\rho_a(b_i,b_i)=1.
\]
Require, in addition, that
\[
R_a=\bigl[\rho_a(b_i,b_j)\bigr]_{i,j=1}^{n_a}
\]
be positive definite.

The fibre metric is then defined by
\[
g_a(\mathbf e_{a,b_i},\mathbf e_{a,b_j})
=\ell_{ab_i}\ell_{ab_j}\rho_a(b_i,b_j).
\]
Equivalently, with respect to the incident-edge basis,
\[
(g_a)_{ij}
=\ell_{ab_i}\ell_{ab_j}\rho_a(b_i,b_j).
\]

\begin{proposition}[Admissibility of the coupled metric]
Let
\[
D_a=\operatorname{diag}
\bigl(\ell_{ab_1},\ldots,\ell_{ab_{n_a}}\bigr).
\]
The coupled Gram matrix is
\[
G_a=D_aR_aD_a.
\]
It is positive definite if and only if \(R_a\) is positive definite.
\end{proposition}

\begin{proof}
The matrix \(D_a\) is invertible because every edge measure is positive. For every non-zero column vector \(x\),
\[
x^{\mathsf T}G_ax = (D_ax)^{\mathsf T}R_a(D_ax).
\]
The claim follows because \(D_ax\ne0\) whenever \(x\ne0\).
\end{proof}

\begin{remark}
The diagonal metric is recovered by taking
\[
\rho_a(b_i,b_j)=\delta_{ij}.
\]
More general positive-definite choices of \(R_a\) introduce metric couplings between incident directions while preserving the physical dimensions inherited from the edge measures. Symmetry and unit diagonal alone are not sufficient for admissibility.
\end{remark}

\subsection{Inner products of fibre-valued incidence cochains} \label{subsec:Metrics_CochainInnerProducts}

The fibre metric induces inner products on spaces of  fibre-valued incidence cochains. To avoid counting the scalar cell weight once for every incident vertex, choose positive incidence weights
\[
\nu_{a|\sigma}>0,
\qquad
\sum_{a\prec\sigma}\nu_{a|\sigma}=1
\]
for every cell \(\sigma\). The uniform choice is
\[
\nu_{a|\sigma} = \frac{1}{\#\operatorname{Vert}(\sigma)}.
\]

Let
\[
X,Y\in C^p(\mathcal{K};F).
\]
For each incidence
\[
(a,\sigma)\in \mathcal{K}_{0,p},
\]
the coefficients
\[
X(a,\sigma),\; Y(a,\sigma)
\]
belong to the same fibre $F_a$ and may therefore be paired through the fibre metric \(g_a\).

\begin{definition}[Vector-valued cochain inner product]
For
\[
X,Y\in C^p(\mathcal{K};F),
\]
define
\[
\langle X,Y\rangle_p = \sum_{(a,\sigma)\in \mathcal{K}_{0,p}} \nu_{a|\sigma} \langle \sigma^\bullet,\sigma^\bullet\rangle_p\, g_a\!\left(X(a,\sigma),Y(a,\sigma)\right),
\]
where
\[
\langle \sigma^\bullet,\sigma^\bullet\rangle_p
\]
denotes the scalar CMC basis norm of the \(p\)-cell \(\sigma\).
\end{definition}

Similarly, the inverse metric induces inner products on covector-valued cochains.

\begin{definition}[Covector-valued cochain inner product]
For
\[
\eta,\xi\in C^p(\mathcal{K};F^*),
\]
define
\[
\langle \eta,\xi\rangle_p
=
\sum_{(a,\sigma)\in \mathcal{K}_{0,p}}
\nu_{a|\sigma}
\langle \sigma^\bullet,\sigma^\bullet\rangle_p\,
g_a^{-1}\!\left(\eta(a,\sigma),\xi(a,\sigma) \right).
\]
\end{definition}

Using the scalar CMC inner product \cite{BERBATOV2026},
\[
\langle \sigma^\bullet,\sigma^\bullet\rangle_p = \frac{1}{2^D\mu(\sigma)} \sum_{\substack{ \chi\in \mathcal{K}_D\\ \tau\in \mathcal{K}_{D-p}\\ \sigma\perp_\chi\tau }} \mu(\tau),
\]
the fibre-valued incidence cochain inner products become
\[
\langle X,Y\rangle_p
=
\sum_{(a,\sigma)\in \mathcal{K}_{0,p}}
\nu_{a|\sigma}
\left(\frac{1}{2^D\mu(\sigma)}
\sum_{\substack{\chi\in \mathcal{K}_D\\ \tau\in \mathcal{K}_{D-p}\\ \sigma\perp_\chi\tau }} \mu(\tau)\right) g_a\!\left(X(a,\sigma), Y(a,\sigma)\right),
\]
and
\[
\langle \eta,\xi\rangle_p
=
\sum_{(a,\sigma)\in \mathcal{K}_{0,p}}
\nu_{a|\sigma}
\left(
\frac{1}{2^D\mu(\sigma)} \sum_{\substack{ \chi\in \mathcal{K}_D\\ \tau\in \mathcal{K}_{D-p}\\ \sigma\perp_\chi\tau }} \mu(\tau)\right) g_a^{-1}\!\left(\eta(a,\sigma),\xi(a,\sigma) \right).
\]
These expressions are positive-definite inner products because the scalar basis weights, incidence weights, and fibre metrics are positive.

\begin{proposition}[Scalar consistency of the incidence inner product]
\label{prop:metric-scalar-consistency}
Let every fibre be \(\mathbb R\) with its standard metric, and let
\[
(\iota\alpha)(a,\sigma)=\alpha(\sigma)
\]
be the vertex-independent incidence lift of a scalar cochain. Then
\[
\langle\iota\alpha,\iota\beta\rangle_p
=\langle\alpha,\beta\rangle_p
\]
for the scalar CMC inner product.
\end{proposition}

\begin{proof}
For each \(\sigma\), the lifted coefficient is independent of \(a\), and
\[
\sum_{a\prec\sigma}\nu_{a|\sigma}=1.
\]
Substitution therefore leaves exactly one copy of the scalar basis weight for every cell.
\end{proof}

\begin{remark}
The inner products do not require transport maps or metric compatibility conditions. They depend on the local fibre metrics, the scalar cochain weights, and the incidence weights \(\nu_{a|\sigma}\).
\end{remark}

\subsection{Induced structures}\label{subsec:Metrics_Induced}

The fibre metric induces corresponding metric structures on the dual fibres and associated tensor spaces.

For vectors,
\[
\langle X,Y\rangle_a = g_a(X,Y), \qquad X,Y\in F_a.
\]
The metric defines the musical maps
\[
g_a^\flat: F_a \longrightarrow F_a^*, \qquad
g_a^\flat(X)(Y) = g_a(X,Y),
\]
and
\[
g_a^\sharp: F_a^* \longrightarrow F_a,
\]
where \(g_a^\sharp\) is the inverse of \(g_a^\flat\).

Equivalently, the inverse metric
\[ 
g_a^{-1}: F_a^*\times F_a^* \longrightarrow \mathbb R
\]
defines the dual inner product
\[
\langle\eta,\xi\rangle_a = g_a^{-1}(\eta,\xi),
\qquad \eta,\xi\in F_a^*.
\]

Similarly, the metric induces inner products on tensor products, exterior powers, and symmetric powers of the fibres.

For fibre endomorphisms,
\[
A,B\in\operatorname{End}(F_a),
\]
the induced inner product is
\[
\langle A,B\rangle_a = \operatorname{tr}\left(A^*B\right),
\]
where \(A^*\) denotes the adjoint with respect to the metric \(g_a\).

\begin{proposition}[Metric decomposition of fibre endomorphisms]
Let
\[
A\in\operatorname{End}(F_a).
\]
Then
\[
A = A_{\mathrm{rot}} + A_{\mathrm{dev}} + A_{\mathrm{sph}},
\]
where
\[
A_{\mathrm{rot}} = \frac12(A-A^\ast),
\]
\[
A_{\mathrm{sph}} = \frac1{\dim(F_a)}
\operatorname{tr}(A)\,I,
\]
and
\[
A_{\mathrm{dev}} = \frac12(A+A^\ast)-A_{\mathrm{sph}}.
\]
This decomposition is unique.
\end{proposition}

\begin{proof}
The adjoint satisfies \((A^\ast)^\ast=A\). Hence \(A_{\mathrm{rot}}^\ast=-A_{\mathrm{rot}}\), while \(\frac12(A+A^\ast)\) is self-adjoint. Since \(\operatorname{tr}(A^\ast)=\operatorname{tr}(A)\), the operator \(A_{\mathrm{dev}}\) is self-adjoint and trace-free. The three terms therefore belong respectively to the metric-skew, self-adjoint trace-free, and scalar subspaces of \(\operatorname{End}(F_a)\). Their sum is \(A\), and the intersection of these three subspaces is trivial, so the decomposition is unique.
\end{proof}

\begin{remark}[Interpretation of the components]
The component
\[
A_{\mathrm{rot}}
\]
is the metric-skew part of the endomorphism and may be interpreted as an infinitesimal rotational component relative to the fibre metric. It is not, in general, a finite rotation. The component
\[
A_{\mathrm{sph}}
\]
acts identically on all directions of the fibre and therefore represents the isotropic part of the transformation. The component
\[
A_{\mathrm{dev}}
\]
is symmetric and trace-free. It measures anisotropic distortion after the removal of both the rotational and isotropic components.
\end{remark}

\begin{remark}
Since
\[
\Omega(a,f)\in\operatorname{End}(F_a),
\]
the orientation-odd curvature admits the above algebraic decomposition. These components are respectively its metric-skew, self-adjoint trace-free, and scalar parts. A physical interpretation of these components requires additional constitutive or kinematic assumptions and is not asserted here.
\end{remark}

\begin{remark}[Endomorphism-valued kinematics and curvature]
Let
\[
u\in C^0(\mathcal{K};F)
\]
be a vector-valued \(0\)-cochain and let
\[
X=\delta^\nabla_{\mathfrak A_0}u
\]
for fixed degree-\(0\) anchoring data.
The associated endomorphism-valued \(0\)-cochain
\[
A_X\in C^0(\mathcal{K};\operatorname{End}(F))
\]
and the curvature endomorphism
\[
\Omega\in C^2(\mathcal{K};\operatorname{End}(F))
\]
both take values in the same fibre algebra
\[
\operatorname{End}(F_a).
\]
Consequently, both admit the above decomposition into rotational, trace-free symmetric, and scalar components.

The endomorphism \(A_X(a)\) describes the local action of the covariant differential of \(u\) on the directions represented by the canonical solder, whereas \(\Omega(a,f)\) is the orientation-odd part of the holonomy around a \(2\)-cell. Whether these two classes of endomorphisms are related through compatibility conditions, evolution equations, or other geometric constraints remains an open question.
\end{remark}

\subsection{Fibre Hodge operators.}

The scalar Hodge-star operator of scalar Combinatorial Mesh Calculus is defined through the scalar cochain inner products and the scalar cup product. The present operators are introduced analogously: the scalar cup product is replaced by the evaluation cup product of Section~\ref{subsec:BundleValued_Evaluation}, while the scalar cochain inner products are replaced by the incidence-cochain inner products of Section~\ref{subsec:Metrics_CochainInnerProducts}. The fibre metric then provides the conversion between vector- and covector-valued coefficients.

For the vector-to-covector Hodge star, the target-side inner product is the covector-valued cochain inner product of Definition~8.4. For the covector-to-vector Hodge star, the target-side inner product is the vector-valued cochain inner product of Definition~8.3.

\begin{definition}[Vector-to-covector Hodge star]
For
\[
0\le p\le D,
\]
the vector-to-covector Hodge star
\[
\star_p^\flat: C^p(\mathcal{K};F) \longrightarrow C^{D-p}(\mathcal{K};F^*)
\]
is the unique operator satisfying
\[
\bigl\langle \star_p^\flat X, \eta \bigr\rangle_{D-p} = (X\smile_{\mathrm{ev}}\eta) [\mathcal{K}]
\]
for every
\[
X\in C^p(\mathcal{K};F), \qquad \eta\in C^{D-p}(\mathcal{K};F^*).
\]
\end{definition}

\begin{definition}[Covector-to-vector Hodge star]
For
\[
0\le p\le D,
\]
the covector-to-vector Hodge star
\[
\star_p^\sharp: C^p(\mathcal{K};F^*) \longrightarrow C^{D-p}(\mathcal{K};F)
\]
is the unique operator satisfying
\[
\bigl\langle \star_p^\sharp\eta, X \bigr\rangle_{D-p} = (\eta\smile_{\mathrm{ev}}X) [\mathcal{K}]
\]
for every
\[
\eta\in C^p(\mathcal{K};F^*), \qquad X\in C^{D-p}(\mathcal{K};F).
\]
\end{definition}

\begin{proposition}[Existence and uniqueness of the Hodge operators]\label{prop:bundle-hodge-existence}
The two Hodge operators in the preceding definitions exist and are unique.
\end{proposition}

\begin{proof}
All incidence-cochain spaces are finite-dimensional, and the inner products of Section~\ref{subsec:Metrics_CochainInnerProducts} are positive definite. For fixed \(X\in C^p(\mathcal K;F)\), the map
\[
\eta\longmapsto
(X\smile_{\mathrm{ev}}\eta)[\mathcal K]
\]
is a linear functional on \(C^{D-p}(\mathcal K;F^*)\). The finite-dimensional Riesz representation theorem therefore gives a unique representing element \(\star_p^\flat X\). Linearity in \(X\) follows from bilinearity of the evaluation cup product. The argument for \(\star_p^\sharp\) is identical.
\end{proof}

\begin{proposition}[Operational form of the fibre Hodge stars]
Let
\[
X\in C^p(\mathcal{K};F), \qquad \eta\in C^p(\mathcal{K};F^*),
\]
and let
\[
(a,c)\in \mathcal{K}_{0,D-p}.
\]
Then the fibre Hodge stars satisfy
\[
(\star_p^\flat X)(a,c) = \frac{1}
{\nu_{a|c}\langle c^\bullet,c^\bullet\rangle_{D-p}} \sum_{\substack{ \chi\in \mathcal{K}_D\\ b\in \mathcal{K}_p\\ b\perp_\chi c\\
N(b,c)=a}} (b^\bullet\smile c^\bullet)[\chi_\bullet]\, g_a^\flat\!\left(X(a,b)\right),
\]
and
\[
(\star_p^\sharp \eta)(a,c) = \frac{1}{\nu_{a|c}\langle c^\bullet,c^\bullet\rangle_{D-p}} \sum_{\substack{\chi\in \mathcal{K}_D\\ b\in \mathcal{K}_p\\ b\perp_\chi c\\
N(b,c)=a}} (b^\bullet\smile c^\bullet)[\chi_\bullet]\, g_a^\sharp\!\left(\eta(a,b)\right).
\]
\end{proposition}

Consequently,
\[
(\star_p^\flat X)(a,c) = \frac{1}{2^D\,\nu_{a|c}\langle c^\bullet,c^\bullet\rangle_{D-p}} \sum_{\substack{\chi\in \mathcal{K}_D\\ b\in \mathcal{K}_p\\ b\perp_\chi c\\
N(b,c)=a}} \operatorname{rel}(\chi,b,c)\,
g_a^\flat\!\left(X(a,b)\right),
\]
and
\[
(\star_p^\sharp \eta)(a,c) = \frac{1} {2^D\,\nu_{a|c}\langle c^\bullet,c^\bullet\rangle_{D-p}} \sum_{\substack{\chi\in \mathcal{K}_D\\ b\in \mathcal{K}_p\\
b\perp_\chi c\\
N(b,c)=a}} \operatorname{rel}(\chi,b,c)\,
g_a^\sharp\!\left(\eta(a,b)\right).
\]
\begin{proof}
The proof follows the scalar argument. Test the defining identity against a covector-valued basis cochain supported on a single
incidence \((a,c)\). Orthogonality of the fibre-valued incidence cochain inner product isolates the coefficient \((\star_p^\flat X)(a,c)\) with the factor
\[
\nu_{a|c}\langle c^\bullet,c^\bullet\rangle_{D-p}.
\]
The evaluation cup product contributes only those pairs \((b,c)\) for which \(b\perp_\chi c\) and \(N(b,c)=a\). Solving the resulting Riesz identity in the fibre \(F_a\) gives the stated formula. For \(\star_p^\sharp\), one tests against a vector-valued basis cochain supported on \((a,c)\); the same argument gives the second formula with \(g_a^\sharp\) replacing \(g_a^\flat\).
\end{proof}

These definitions are direct analogues of the scalar Hodge-star construction. The scalar pairing between complementary cochains is replaced by the
evaluation pairing between fibres and dual fibres, while the fibre metric provides the identification between vector-valued and covector-valued
coefficients.

In particular, for a three-dimensional complex,
\[
\star_1^\flat:
C^1(\mathcal K;F)
\longrightarrow
C^2(\mathcal K;F^*)
\]
converts directional quantities into dual-valued quantities.
Likewise,
\[
\star_2^\sharp:
C^2(\mathcal K;F^*)
\longrightarrow
C^1(\mathcal K;F)
\]
converts dual-valued quantities back into directional ones.

\begin{remark}
The fibre Hodge star contains the scalar CMC cell weights, the incidence weights, and the fibre metric. These data define a geometric Riesz map. Additional constitutive information may be incorporated through a separately specified positive or otherwise physically admissible fibre operator; it is not implied by the metric alone.
\end{remark}

\begin{remark}
The fibre Hodge operators combine naturally with the evaluation cup product. For example,
\[
(\star_p^\flat X)\smile_{\mathrm{ev}}X \in
C^D(\mathcal{K})
\]
is a scalar \(D\)-cochain obtained by pairing a vector-valued cochain with its associated dual-valued cochain. Such quantities provide natural combinatorial analogues of energies, powers, virtual works, and related bilinear forms.
\end{remark}

\subsection{Metric compatibility}
\label{subsec:Metrics_Compatibility}

The directed transport structure and the metric are introduced independently.

\begin{definition}[Metric-compatible directed transport]
A directed edge map \(P_{ab}:F_a\to F_b\) is metric compatible if
\[
g_b(P_{ab}X,P_{ab}Y)=g_a(X,Y),
\]
for all
\[
X,Y\in F_a.
\]
A directed transport structure is metric compatible on full fibres if this condition holds for every oriented \(1\)-cell.
\end{definition}

\begin{proposition}[Rank restrictions]
\label{prop:metric-rank-restrictions}
If \(P_{ab}:F_a\to F_b\) is metric compatible, then it is injective and
\[
\dim F_a\le\dim F_b.
\]
If both \(P_{ab}\) and \(P_{ba}\) are metric compatible on full fibres, then
\[
\dim F_a=\dim F_b
\]
and both maps are isomorphisms. Metric compatibility alone does not imply
\[
P_{ba}=P_{ab}^{-1}.
\]
\end{proposition}

\begin{proof}
If \(P_{ab}X=0\), then
\[
g_a(X,X)=g_b(P_{ab}X,P_{ab}X)=0.
\]
Positive definiteness gives \(X=0\), so \(P_{ab}\) is injective, and the dimension inequality follows. Applying the same argument to \(P_{ba}\) gives the reverse inequality. The two maps may nevertheless be distinct isometries unless reversibility is imposed separately.
\end{proof}

When transport maps act only on locality-dependent subspaces, the same condition may be imposed on those subspaces. In particular, full-fibre metric compatibility in both directions is unavailable across an edge whose endpoint fibres have unequal dimensions. Metric compatibility and reversibility are independent conditions.

\begin{remark}[Metric-compatible holonomy and finite defects]
Suppose that the transport is metric compatible along every edge of the boundary of the oriented \(2\)-cell \(f\), with transports defined on the relevant full fibres. Then the corresponding holonomy
\[
\mathcal{R}(a,f):F_a\longrightarrow F_a
\]
preserves the fibre metric,
\[
g_a\left(\mathcal{R}(a,f)X,\mathcal{R}(a,f)Y\right) = g_a(X,Y)
\]
for all
\[
X,Y\in F_a.
\]
Equivalently,
\[
\mathcal{R}(a,f)^{\!*}\mathcal{R}(a,f) = \operatorname{Id}_{F_a},
\]
where \(\mathcal{R}(a,f)^{\!*}\) denotes the adjoint with respect to \(g_a\).

The finite holonomy defect
\[
\Delta(a,f) = \mathcal{R}(a,f)-\operatorname{Id}_{F_a}
\]
is not generally skew-adjoint. Indeed,
\[
\Delta(a,f)^{\!*} = \mathcal{R}(a,f)^{-1}-\operatorname{Id}_{F_a},
\]
which need not equal
\[
-\Delta(a,f).
\]
If, in addition, the transport is reversible along the boundary, then
\[
\mathcal{R}(a,-f)=\mathcal{R}(a,f)^{-1}
\]
and the orientation-odd curvature becomes
\[
\Omega(a,f) = \frac12\left( \mathcal{R}(a,f)-\mathcal{R}(a,f)^{-1} \right).
\]
It is then skew-adjoint:
\[
\Omega(a,f)^{\!*}=-\Omega(a,f).
\]
\end{remark}

\begin{remark}[Infinitesimal curvature generators]
Suppose the holonomy lies in a neighbourhood of the identity on which a real logarithm is defined and the chosen logarithm is compatible with the adjoint operation. One may then introduce the infinitesimal generator
\[
K(a,f) = \log \mathcal{R}(a,f).
\]
For metric-compatible holonomy in this neighbourhood, the generator is skew-adjoint,
\[
K(a,f)^{\!*} = -K(a,f).
\]
Alternatively, the skew-adjoint part of the finite holonomy may be defined by
\[
\mathcal{K}_{\mathrm{skew}}(a,f) = \frac{1}{2} \left(\mathcal{R}(a,f)-\mathcal{R}(a,f)^{-1}\right).
\]
For reversible transport this is exactly the orientation-odd curvature \(\Omega(a,f)\). By contrast, \(K(a,f)\) is a logarithmic generator and \(\Delta(a,f)=\mathcal{R}(a,f)-\operatorname{Id}_{F_a}\) is the full finite holonomy defect. These three objects agree only to leading order near the identity.
\end{remark}

\subsection{Geometric data}
\label{subsec:Metrics_Geometry}

The constructions developed so far introduce five core geometric ingredients,
\[
(\mathcal{K},F,\theta,P,g),
\]
consisting of
\begin{itemize}
\item the computational complex \(\mathcal{K}\);
\item the local fibres \(F\);
\item the solder form \(\theta\);
\item the directed transport structure \(P\);
\item the fibre metric \(g\).
\end{itemize}
These five ingredients do not, by themselves, determine the weighted incidence operators or the incidence inner products. Those constructions also require
\[
\mathfrak A=\{\mathfrak A_p\}_{p=0}^{D-1}
\qquad\text{and}\qquad
\nu=\{\nu_{a|\sigma}\},
\]
respectively. Thus, a fully specified operational model has the augmented data
\[
(\mathcal K,F,\theta,P,g;\mathfrak A,\nu).
\]
Different choices give rise to different incidence-based combinatorial geometries and operators. Compatibility conditions, including metric compatibility and Cartan-type relations between torsion, curvature, and the weighted incidence operators, may be imposed according to the intended mathematical or physical setting.

\begin{remark}[Geometric synthesis]
The core \(5\)-tuple
\[
(\mathcal{K},F,\theta,P,g)
\]
combines an incidence-generated family of local coefficient spaces with a solder form, directed edge transports, and fibre metrics. It resembles local-system and cellular-sheaf constructions, but it is not identified with either structure without the additional constant-rank, invertibility, cell-stalk, and compatibility hypotheses described in Section~\ref{subsec:Connections_Definition}. The operators \((\delta^\nabla_{\mathfrak A_p},\star^\flat,\star^\sharp)\) additionally depend on \(\mathfrak A\) or \(\nu\) and should be regarded as structures of the present incidence-based framework rather than as standard sheaf operators.
\end{remark}

\section{Locality-Dependent Operator Algebras}
\label{sec:Locality}

Locality entered the preceding constructions through the incidence-generated fibres, the directed transport structure, and the anchoring data
\[
\mathfrak A_p=(S_p,\Gamma_p,w^p)
\]
of the weighted covariant incidence operators. This section isolates a second notion of locality arising entirely at the level of operators acting on scalar \(0\)-cochains. The resulting hierarchy forms a filtered associative algebra and therefore a filtered Lie algebra under commutators. No additional homotopy-Lie structure is required for the results established here.

\subsection{Graph distance and the locality filtration}
\label{subsec:Locality_Hierarchy}

\begin{definition}[Combinatorial distance]
For vertices \(a,b\in\mathcal K_0\), let \(d(a,b)\) be the minimum number of \(1\)-cells in a path of the \(1\)-skeleton joining \(a\) to \(b\), with \(d(a,b)=\infty\) when the vertices lie in different connected components. For \(r\ge0\), define
\[
N_r(a)=\{\,b\in\mathcal K_0:d(a,b)\le r\,\}.
\]
\end{definition}

\begin{definition}[Locality-\(r\) operators]
The locality-\(r\) operator space is
\[
\mathfrak D^{(r)}(\mathcal K) = \left\{
A\in\operatorname{End}\bigl(C^0(\mathcal K)\bigr):
(Af)(a)\text{ depends only on }f|_{N_r(a)}
\right\}.
\]
\end{definition}

Equivalently, after choosing the vertex basis of \(C^0(\mathcal K)\), an operator \(A\) belongs to \(\mathfrak D^{(r)}(\mathcal K)\) precisely when its matrix entries satisfy
\[
A_{ab}=0
\qquad\text{whenever}\qquad
d(a,b)>r.
\]

\begin{proposition}[Filtered operator algebra]
\label{prop:locality-filter}
For \(r,s\ge0\),
\[
\mathfrak D^{(r)}
\subseteq
\mathfrak D^{(r+1)},
\qquad
\mathfrak D^{(r)}\mathfrak D^{(s)}
\subseteq
\mathfrak D^{(r+s)},
\]
and therefore
\[
[\mathfrak D^{(r)},\mathfrak D^{(s)}]
\subseteq
\mathfrak D^{(r+s)}.
\]
Thus \(\{\mathfrak D^{(r)}\}_{r\ge0}\) is a filtration of an associative operator algebra and, under commutators, of a Lie algebra.
\end{proposition}

\begin{proof}
The first inclusion follows from \(N_r(a)\subseteq N_{r+1}(a)\). If \((AB)_{ac}\ne0\), then there is a vertex \(b\) with \(A_{ab}\ne0\) and \(B_{bc}\ne0\). Hence
\[
d(a,c)\le d(a,b)+d(b,c)\le r+s,
\]
which proves \(AB\in\mathfrak D^{(r+s)}\). The commutator statement follows by applying the same result to \(AB\) and \(BA\).
\end{proof}

\begin{remark}[Finite stabilisation]
If the \(1\)-skeleton is connected with diameter \(R\), then
\[
\mathfrak D^{(R)}(\mathcal K) = \operatorname{End}\bigl(C^0(\mathcal K)\bigr).
\]
For a disconnected complex, the filtration stabilises instead at the algebra of operators that are block diagonal with respect to the connected components.
\end{remark}

\subsection{Operator representation of combinatorial vector fields}
\label{subsec:Locality_Operators}

Let
\[
X\in C^0(\mathcal K;F).
\]
Using the incidence-generated fibre basis, write
\[
X(a) = \sum_{b\in\operatorname{Adj}(a)} X_a^{\,b}\,\mathbf e_{a,b}.
\]
For adjacent vertices \(a,b\), let \(e_{ab}\) denote the underlying \(1\)-cell and define the outward difference
\[
\nabla_{ab}f = f(b) - f(a) = -\varepsilon(e_{ab},a)(\delta f)(e_{ab}).
\]
This expression is independent of the chosen global orientation of \(e_{ab}\).

\begin{definition}[Operator induced by a vector field]
Define
\[
L_X:C^0(\mathcal K)\longrightarrow C^0(\mathcal K)
\]
by
\[
(L_Xf)(a) = \sum_{b\in\operatorname{Adj}(a)} X_a^{\,b}\bigl(f(b)-f(a)\bigr).
\]
\end{definition}

Every \(L_X\) has locality radius at most one and annihilates constant cochains. Let \(\mathbf 1\) denote the constant cochain with value \(1\), and define
\[
\mathfrak L^{(r)} = \{\,A\in\mathfrak D^{(r)}:A\mathbf1=0\,\}.
\]

\begin{proposition}[Characterisation of radius-one vector-field operators]
\label{prop:locality-vector-image}
The map
\[
X\longmapsto L_X
\]
is a linear isomorphism from \(C^0(\mathcal K;F)\) onto \(\mathfrak L^{(1)}\).
\end{proposition}

\begin{proof}
The matrix of \(L_X\) satisfies
\[
(L_X)_{ab}=X_a^{\,b} \quad\text{for }b\in\operatorname{Adj}(a),
\qquad
(L_X)_{aa} = -\sum_{b\in\operatorname{Adj}(a)}X_a^{\,b},
\]
with all other entries equal to zero. Hence \(L_X\in\mathfrak L^{(1)}\), and the off-diagonal entries recover the coefficients \(X_a^{\,b}\), proving injectivity.

Conversely, if \(A\in\mathfrak L^{(1)}\), set
\[
X_a^{\,b}=A_{ab}
\qquad
(b\in\operatorname{Adj}(a)).
\]
The identity \(A\mathbf1=0\) gives
\[
A_{aa} = -\sum_{b\in\operatorname{Adj}(a)}A_{ab},
\]
so the resulting operator \(L_X\) equals \(A\).
\end{proof}

\begin{proposition}[Filtered algebra of constant-annihilating operators]
\label{prop:locality-L-filter}
For \(r,s\ge0\),
\[
\mathfrak L^{(r)}\mathfrak L^{(s)}
\subseteq
\mathfrak L^{(r+s)},
\qquad
[\mathfrak L^{(r)},\mathfrak L^{(s)}]
\subseteq
\mathfrak L^{(r+s)}.
\]
\end{proposition}

\begin{proof}
The locality bound follows from Proposition~\ref{prop:locality-filter}. If \(A\mathbf1=B\mathbf1=0\), then
\[
AB\mathbf1=A(B\mathbf1)=0,
\]
and similarly \(BA\mathbf1=0\). The two inclusions follow.
\end{proof}

\subsection{Leibniz defect and the pointwise algebra}
\label{subsec:Locality_Leibniz}

The scalar \(0\)-cochains carry the pointwise product
\[
(fg)(a) = f(a)g(a).
\]
For \(f,g\in C^0(\mathcal K)\), define
\[
\Lambda_X(f,g) = L_X(fg) - (L_Xf)g - f(L_Xg).
\]

\begin{proposition}[Leibniz defect]
\label{prop:leibniz-defect}
For every vertex \(a\),
\[
\Lambda_X(f,g)(a)
=
\sum_{b\in\operatorname{Adj}(a)}
X_a^{\,b}
\bigl(f(b)-f(a)\bigr)
\bigl(g(b)-g(a)\bigr).
\]
\end{proposition}

\begin{proof}
For each adjacent pair \(a,b\),
\[
\begin{aligned}
f(b)g(b)-f(a)g(a)
={}&
\bigl(f(b)-f(a)\bigr)g(a)
+
f(a)\bigl(g(b)-g(a)\bigr)
\\
&+
\bigl(f(b)-f(a)\bigr)
\bigl(g(b)-g(a)\bigr).
\end{aligned}
\]
Multiplication by \(X_a^{\,b}\) and summation over \(b\in\operatorname{Adj}(a)\)
gives the result.
\end{proof}

The failure of the Leibniz rule is not merely generic: the pointwise algebra on a finite vertex set admits no non-zero derivations.

\begin{proposition}[Derivations of the pointwise algebra]
\label{prop:zero-derivations}
Every derivation
\[
D:C^0(\mathcal K)\longrightarrow C^0(\mathcal K)
\]
of the pointwise algebra is zero.
\end{proposition}

\begin{proof}
For each vertex \(a\), let \(q_a\) be its indicator cochain. Since \(q_a^2=q_a\), the Leibniz rule gives
\[
D(q_a) = D(q_a^2)=2q_aD(q_a).
\]
At every vertex different from \(a\), the right-hand side vanishes, while at \(a\) the identity reads
\[
D(q_a)(a) = 2D(q_a)(a).
\]
Thus \(D(q_a)=0\). The indicator cochains form a basis of \(C^0(\mathcal K)\), so \(D=0\).
\end{proof}

\begin{corollary}
The operator \(L_X\) is a derivation of the pointwise algebra if and only if \(X=0\).
\end{corollary}

\begin{proof}
If \(L_X\) is a derivation, Proposition~\ref{prop:zero-derivations} gives \(L_X=0\), and Proposition~\ref{prop:locality-vector-image} then gives \(X=0\). The converse is immediate.
\end{proof}

\begin{remark}[Finite-difference origin of the Leibniz defect]
The defect term
\[
\Lambda_X(f,g)
\]
is quadratic in the finite differences
\[
f(b)-f(a),
\qquad
g(b)-g(a).
\]
It therefore vanishes whenever either function is locally constant along the edges contributing to \(L_X\). The obstruction to the Leibniz rule is thus a finite-locality effect arising from discrete differences rather than an artefact of a particular choice of coefficients.
\end{remark}

\subsection{Commutators and enlargement of locality}
\label{subsec:Locality_Commutators}

Proposition~\ref{prop:locality-L-filter} gives
\[
[L_X,L_Y]\in\mathfrak L^{(2)}
\]
for all \(X,Y\in C^0(\mathcal K;F)\). The radius-two bound is sharp in general.

\begin{example}[A distance-two commutator]
\label{ex:distance-two-commutator}
Let the \(1\)-skeleton contain the path
\[
a_0 - a_1 - a_2.
\]
Choose vector fields with the only non-zero coefficients
\[
X_{a_0}^{\,a_1}=1,
\qquad
Y_{a_1}^{\,a_2}=1.
\]
Then
\[
(L_Xf)(a_0)=f(a_1)-f(a_0),
\qquad
(L_Yf)(a_1)=f(a_2)-f(a_1),
\]
and all other rows of these two operators vanish. Consequently,
\[
([L_X,L_Y]f)(a_0) = f(a_2) - f(a_1).
\]
This value depends on the distance-two vertex \(a_2\), so
\[
[L_X,L_Y]\notin\mathfrak L^{(1)}
\]
for this example.
\end{example}

Thus, radius-one vector-field operators are not closed under commutators on a general \(1\)-skeleton. Their commutators, nevertheless, close exactly in the filtered family:
\[
[\mathfrak L^{(r)},\mathfrak L^{(s)}]
\subseteq\mathfrak L^{(r+s)}.
\]

\subsection{Associated graded algebra}
\label{subsec:Locality_Graded}

Set \(\mathfrak D^{(-1)}=\{0\}\) and define
\[
\operatorname{gr}\mathfrak D
=
\bigoplus_{r\ge0}
\mathfrak D^{(r)}/\mathfrak D^{(r-1)}.
\]
The product and commutator respect the filtration and therefore induce a graded associative product and a graded Lie bracket on \(\operatorname{gr}\mathfrak D\). Likewise,
\[
\operatorname{gr}\mathfrak L
=
\bigoplus_{r\ge0}
\mathfrak L^{(r)}/\mathfrak L^{(r-1)}
\]
inherits a graded Lie bracket. The Jacobi identity is exact because these brackets are induced by ordinary operator commutators.

\subsection{Remaining questions}
\label{subsec:Locality_Open}

\begin{problem}
Characterise the coefficient cancellations and graph structures for which
\[
[L_X,L_Y]\in\mathfrak L^{(1)}
\]
despite the general radius-two bound.
\end{problem}

\begin{problem}
Relate the operator filtration \(\mathfrak L^{(r)}\) to the locality prescriptions encoded by the anchor sets, paths, and weights of the weighted covariant incidence operators \(\delta^\nabla_{\mathfrak A_p}\).
\end{problem}

\begin{problem}
Define higher-degree contraction and Lie-type operations compatible with the incidence-cochain typing, the evaluation products, and the locality filtration.
\end{problem}

\section{Summary and Outlook}\label{sec:Summary}

\subsection{Summary of the construction}\label{subsec:Summary_Construction}

The present work develops a fibre-valued extension of Combinatorial Mesh Calculus, built upon the five core structures
\[
(\mathcal{K},F,\theta,P,g),
\]
together with anchoring data \(\mathfrak A\) and incidence weights \(\nu\) for those operators that require aggregation over vertex--cell incidences.

Starting from the scalar theory, local fibres are constructed from the incidence structure of the computational complex and provide the coefficient spaces for fibre-valued incidence cochains. Dual fibres and the evaluation cup product provide natural pairings between kinematic and force-like quantities and permit the construction of scalar energetic cochains. Directed transports compare neighbouring fibres and generate path transport and holonomy. Together with explicit anchor sets, paths, and normalising weights, they also define degree-raising covariant incidence operators.

The resulting framework provides combinatorial analogues of many of the structures appearing in differential geometry and geometric mechanics while remaining entirely within the setting of cell complexes. This provides the mathematical scaffold for formulating mechanics and transport directly on evolving cellular organisation rather than on an intermediate continuum.

\subsection{Comparison with smooth differential geometry}\label{subsec:Summary_Comparison}

Although many constructions introduced in this work resemble their smooth counterparts, their origins and operational properties are often substantially different.

\paragraph{Cell complexes versus manifolds.}

The starting point is an explicit cell complex rather than a smooth manifold. Topology is represented directly by incidences between cells, and no differentiable structure is assumed.

\paragraph{Hodge stars and adjoint operators.}

In smooth differential geometry, the Hodge star and codifferential are both derived from a Riemannian metric. In scalar CMC, the two constructions are introduced independently.

The Hodge-star operator is defined through the inner product and cup product, whereas the adjoint coboundary is defined solely from the inner product. Consequently, the identity 
\[
\delta^{\star} = \star\,\delta\,\star
\]
need not hold exactly.

One possible reason is that the discrete inner product, cup product, and Hodge-star operator originate from distinct combinatorial constructions, whose mutual compatibility is weaker than in the smooth theory.

\paragraph{Cup products.}

The cup product of scalar CMC is generally non-associative. This differs from the smooth exterior product, which is associative. The origin of this behaviour appears to be combinatorial. Products are assembled through incidences of finite cells rather than through pointwise multiplication on an underlying smooth algebra. As a consequence, associativity need not be preserved exactly at the discrete level.

\paragraph{Fibres and vector fields.}

The fibres introduced here may be viewed as combinatorial analogues of tangent spaces \cite{Nakahara2003, Needham2021}, but they are constructed differently. Rather than arising from a manifold structure, the fibre
\[
F_a
=
\operatorname{span}
\{\,\mathbf e_{a,b}:b\in\operatorname{Adj}(a)\,\}
\]
is generated directly by the local incidence structure of the complex.

Similarly, combinatorial vector fields are sections of these fibres. They possess a geometric interpretation analogous to vector fields in the smooth setting, but their operator representation is different.

The induced operators possess the explicit quadratic Leibniz defect of Proposition~\ref{prop:leibniz-defect}. More strongly, the finite-dimensional pointwise algebra admits no non-zero derivations, so \(L_X\) is a derivation precisely when \(X=0\).

\paragraph{Connections, torsion and curvature.}

In the smooth theory, curvature and torsion are commonly introduced through differential operators and Cartan's structure equations \cite{Nakahara2003, Needham2021}.

The present framework adopts the opposite direction. Transport maps are introduced first between local fibres. Torsion is then defined through transported solder forms, and curvature through oriented holonomies. Weighted covariant incidence operators are introduced afterwards.

This operational approach naturally leads to the question of whether differential and transport-based constructions eventually become equivalent.

\paragraph{Locality.}

Locality plays a fundamentally different role from that encountered in smooth differential geometry.

Transport maps, weighted incidence operators, and induced operator algebras depend on chosen anchor sets, paths, and combinatorial neighbourhoods. As a result, locality becomes an explicit structural parameter rather than an implicitly infinitesimal property of a smooth manifold.

\paragraph{Relation to cellular sheaves.}
The family \(F\) and the edge maps \(P\) resemble the vertex data used in cellular sheaves and local systems \cite{Shepard1985cellular, Curry2014sheaves}, but the present incidence-based construction is not automatically an object of either category. Its proposed contribution is instead the combination of incidence-generated, possibly variable-rank vertex fibres with a solder form \(\theta\), dual evaluation pairings \(\smile_{\mathrm{ev}}\), bundle Hodge operators \((\star^\flat,\star^\sharp)\), and locality-dependent directed transports. Establishing a precise functorial relation with cellular sheaves remains a separate mathematical problem.

\subsection{Open questions}
\label{subsec:Summary_Open}

The present framework remains incomplete, and several questions require further investigation.

\begin{enumerate}
\item Determine admissibility conditions for transport maps and their relation to locality, rank variation, metric compatibility, and physical interpretation.
\item Extend Proposition~\ref{prop:torsion-recovery} to a full Cartan-type structure equation and determine how the holonomy defects arise from connection data.
\item Construct a rigorously typed exterior covariant calculus on higher tensor-valued incidence cochains and define its action products.
\item Construct a full combinatorial Cartan calculus, including higher-degree interior products and Lie-type derivatives.
\item Characterise the associated graded locality algebra and relate its filtration to the anchoring data of the weighted incidence operators.
\item Identify natural geometric invariants derived from the curvature, torsion and transport structures.
\item Formulate coarse-graining and homogenization projections relating the variable-dimensional local fibres $F_a \simeq \mathbb{R}^{n_a}$ to effective continuous tangent bundles $TM \cong \mathbb{R}^D$ in the macroscopic continuum limit.
\end{enumerate}

These questions suggest that explicit cellular fabrics possess geometric and algebraic structures that are substantially richer than those captured by scalar cochain theories alone. Understanding the interaction between topology, transport, locality, and metric structure therefore remains a central problem for the further development of fibre-valued Combinatorial Mesh Calculus.

\appendix

\section{Notation}

\subsection*{Complexes and cells}

\begin{tabular}{ll}
$\mathcal{M}$ & Physical complex (simple polytopes) \\
$\mathcal{K}$ & Computational complex (Forman subdivision of $\mathcal{M},$ cubical) \\
$\mathcal{K}_p$ & Set of $p$-cells of $\mathcal{K}$ \\
$\mathcal{K}_{0,p}$ & Set of incidences $(a,\sigma)$ with $a\in \mathcal{K}_0$ and $a\prec \sigma$ \\
$\sigma,\tau,\chi$ & Generic cells of $\mathcal{K}$ \\
$a,b,c$ & Generic vertices ($0$-cells) \\
$e$ & Generic $1$-cell \\
$f$ & Generic $2$-cell \\
$\sigma\prec\tau$ & Face relation \\
$\varepsilon(\tau,\sigma)$ & Relative orientation (incidence number) \\
$\operatorname{Adj}(a)$ & Vertices adjacent to \(a\) in the \(1\)-skeleton \\
\end{tabular}

\subsection*{Fibres and fibre-valued incidence cochains}

\begin{tabular}{ll}
$F_a$ & Fibre attached to vertex $a$ \\
$F_a^\ast$ & Dual fibre \\
$\mathbf e_{a,b}$ & Fibre generator associated with oriented $1$-cell from $a$ to $b$ \\
$\mathbf e^{a,b}$ & Dual basis element in $F_a^*$\\
$F$ & Family of fibres $\{F_a\}_{a\in \mathcal{K}_0}$ \\
$C^p(\mathcal{K};F)$ & Vector-valued $p$-cochains \\
$C^p(\mathcal{K};F^\ast)$ & Covector-valued $p$-cochains \\
$C^p(\mathcal{K};\operatorname{End}(F))$ & Endomorphism-valued $p$-cochains \\
$\mathcal I^p(\mathcal K;W)$ & Unconstrained \(W\)-valued incidence functions \\
$\mathcal I^p_\pm(\mathcal K;W)$ & Orientation-even/odd incidence functions \\
$X$ & Vector-valued 0-cochain (combinatorial vector field) \\
$\eta$ & Covector-valued 0-cochain (combinatorial covector field)\\
$A$ & Endomorphism-valued 0-cochain \\
\end{tabular}

\subsection*{Cup products and pairings}

\begin{tabular}{ll}
$\smile_{\mathrm{ev}}$ & Evaluation cup product \\
$N(b,c)$ & Common vertex of perpendicular cells $b$ and $c$ \\
$\perp_{p,q}a$ & Set of perpendicular $p$- and $q$-subcells of $a$ \\
$\operatorname{rel}(a,b,c)$ & Relative orthogonal orientation \\
\end{tabular}

\subsection*{Connections and transport}

\begin{tabular}{ll}
$P_{ab}$ & Directed transport map $F_a\rightarrow F_b$ \\
$P_\gamma$ & Transport along path $\gamma$ \\
$P^\ast_{ab}$ & Dual transport \\
$\operatorname{Vert}(\rho)$ & Vertex set of a cell \(\rho\) \\
$S_p(a,\tau)$ & Anchor set for the incidence operator \\
$w^p_{a,\tau}(b,\sigma)$ & Anchor weight \\
$\mathfrak A_p$ & Degree-\(p\) anchoring data \\
$\gamma^a_{\tau,b}$ & Chosen anchor path from \(b\) to \(a\) \\
$\delta^\nabla_{\mathfrak A_p}$ & Weighted covariant incidence operator \\
\end{tabular}

\subsection*{Torsion and curvature}

\begin{tabular}{ll}
$\theta$ & Solder form \\
$\Theta$ & Torsion cochain \\
$\mathcal{R}(a,f)$ & Holonomy operator \\
$\Delta$ & Finite holonomy defect \\
$\Omega$ & Orientation-odd curvature cochain \\
\end{tabular}

\subsection*{Metric structures}

\begin{tabular}{ll}
$g_a$ & Fibre metric on $F_a$ \\
$e_{ab}$ & \(1\)-cell joining adjacent vertices \(a\) and \(b\) \\
$\ell_{ab}$ & Edge measure \(\mu(e_{ab})\) \\
$R_a$ & Positive-definite directional coupling matrix \\
$\nu_{a|\sigma}$ & Normalised incidence weight \\
$g_a^{-1}$ & Dual metric on $F_a^\ast$ \\
$g_a^\flat$ & Musical isomorphism $F_a\rightarrow F_a^\ast$ \\
$g_a^\sharp$ & Inverse musical isomorphism $F_a^\ast\rightarrow F_a$ \\
$\star_p^\flat$ & Vector-to-covector Hodge star \\
$\star_p^\sharp$ & Covector-to-vector Hodge star \\
\end{tabular}

\subsection*{Locality-dependent algebra}

\begin{tabular}{ll}
$\nabla_{ab}f$ & Outward difference \(f(b)-f(a)\) \\
$L_X$ & Operator induced by a combinatorial vector field \\
$\Lambda_X$ & Leibniz defect \\
$\mathfrak D^{(r)}$ & Locality-$r$ operator space \\
$\mathfrak L^{(r)}$ & Locality-\(r\) operators annihilating constants \\
$N_r(a)$ & Combinatorial $r$-neighbourhood of $a$ \\
\end{tabular}

\bibliographystyle{unsrt}
\bibliography{references}

\end{document}